\documentclass[12pt]{amsart}
\usepackage{amsmath,amsthm,amsfonts,amssymb,amscd,euscript}
\usepackage{enumerate,enumitem}
\usepackage{color}
\usepackage{graphics}
\usepackage{tikz}
\usepackage{mathtools}
\usepackage{tkz-graph}
\usetikzlibrary{automata,arrows}
\usepackage{graphicx}
\usepackage{todonotes}
\usepackage{hyperref}
\usepackage{marginnote}

\usetikzlibrary{decorations.pathmorphing}
\usetikzlibrary{patterns}

\input{xy}
\xyoption{all}
\newlength{\defbaselineskip}
\theoremstyle{plain}
\newtheorem{theorem}{Theorem}[section]

\newtheorem{lemma}[theorem]{Lemma}

\theoremstyle{definition}
\newtheorem{construction}[theorem]{Construction}
\newtheorem{definition}[theorem]{Definition}

\newtheorem{remark}[theorem]{Remark}
\newtheorem{example}[theorem]{Example}

\newcommand{\musep}{\mu_{\text{Ind}}}
\newcommand{\mucycle}{\mu_{\text{cycle}}}

\DeclareMathOperator{\mut}{Mut}

\begin{document}
%\definecolor{darksagegreen}{rgb}{0.09, 0.45, 0.27}
\definecolor{darksagegreen}{RGB}{0, 105, 60} 
\author{Matthew R. Mills}
\thanks{\email{matthew.mills@huskers.unl.edu}Supported by University of Nebraska - Lincoln, and by NSA grant H98230-14-1-0323.}
\title{Maximal green sequences for quivers of finite mutation type}
% put your affiliation here, not your full address. If you like to give
% away your email address, put it in the \thanks as above.
\address{Department of Mathematics, University of Nebraska - Lincoln, Lincoln, NE, USA}

\begin{abstract}
In general, the existence of a maximal green sequence is not mutation invariant. In this paper we show that it is in fact mutation invariant for cluster quivers of finite mutation type. In particular, we show that a mutation finite cluster quiver has a maximal green sequence unless it arises from a once-punctured closed marked surface, or one of the two quivers in the mutation class of $\mathbb{X}_7$. We develop a procedure to explicitly find maximal green sequences for cluster quivers associated to arbitrary triangulations of closed marked surfaces with at least two punctures. As a corollary, it follows that any triangulation of a marked surface with boundary has a maximal green sequence. We also compute explicit maximal green sequences for exceptional quivers of finite mutation type. 
\end{abstract}

\maketitle

\section{Introduction}
\label{sec:in}
Cluster algebras were introduced by Fomin and Zelevinsky in \cite{cluster1}. Cluster algebras have become an important tool in the study of many areas of mathematics and mathematical physics. They play a role in the study of Teichm\"{u}ller theory, canonical bases, total positivity, Poisson-Lie groups, Calabi-Yau algebras, noncommutative Donaldson-Thomas invariants, scattering amplitudes, and representations of finite dimensional algebras.

One very important property of a quiver associated to a cluster algebra is whether or not it has a maximal green sequence. Quiver mutation is a transformation of a quiver, determined by a choice of a vertex of the quiver, into a new quiver. A maximal green sequence is a certain sequence of quiver mutations given by a sequence of vertices of the quiver. 
The idea of maximal green sequences of cluster mutations was introduced by Keller in \cite{keller}. He explored quantum dilogarithm identities by utilizing these sequences in the explicit computation of noncommutative Donaldson-Thomas invariants of quivers which were introduced by Kontsevich and Soibelman in \cite{kontsevich}. If a quiver with potential has a maximal green sequence, then its associated Jacobi algebra is finite dimensional \cite{brustle,keller2}. In \cite{amiot} an explicit construction of a cluster category from a quiver with potential whose Jacobian algebra is finite dimensional is given. 

The first result in this paper focuses on the existence of maximal green sequences for quivers that are associated to triangulations of surfaces. 

\begin{theorem}\label{thm:main}
Suppose that a marked surface $\Sigma$ is not once-punctured and closed, that is, $\Sigma$ is
\begin{enumerate}
\item[(A)] of genus at least one with at least two punctures;
\item[(B)] of genus zero with at least four punctures;
\item[(C)] or of arbitrary genus with at least one boundary component.\end{enumerate}
Then for any triangulation of $\Sigma$ there exists a maximal green sequence for the associated quiver. 
\end{theorem}
The proof of Theorem \ref{thm:main} for $(A)$ and $(B)$ is an explicit construction of maximal green sequences for these surfaces. The construction is given in Section \ref{sec:proof}. The proof for $(C)$ follows from the previous cases together with a theorem of Muller that we recall here as Theorem \ref{thm:subquiver}. The author had originally provided a construction for the case of higher genus surfaces and discussed surfaces with boundary in an extended abstract \cite{mills}. The construction given in this paper is a refinement of the one given there that also applies to punctured spheres. 
This construction of maximal green sequences requires many choices so we in fact get many different maximal green sequences for these quivers. It follows from the work of Keller \cite{keller} that our work here provides many quantum dilogorithm identities.

It is straightforward from the definition of quiver mutation (formally stated in Definition \ref{def:mutation}) that mutation imposes an equivalence relation on the set of all quivers. For a quiver $Q$ we let $\mut(Q)$ denote the equivalence class of $Q$ under this relation.

Muller showed that in general the existence of a maximal green sequence is not mutation invariant \cite{muller}. It is already known that every quiver of type $\mathbb{A}, \mathbb{D},$ and $\mathbb{E}$ has a maximal green sequence \cite{brustle}. It was shown by Ladkani that quivers associated to once-punctured closed surfaces of genus at least one do not admit maximal green sequences \cite{ladkani}. The existence of maximal green sequences for specific triangulations of various marked surfaces has been shown in many papers \cite{alim,bucher,bm}. In \cite{garver} Garver and Musiker give a combinatorial approach to construct maximal green sequences for type $\mathbb{A}$ quivers, which are exactly the quivers associated to triangulations of unpunctured disks. Cormier \textit{et al}., give an explicit construction of minimal length maximal green sequences for this case in \cite{cormier}.

In \cite{brustle2} Br\"ustle, Hermes, Igusa, and Todorov use semi-invariants to prove two conjectures about maximal green sequences. 
One particularly usefull result from this paper is the Rotation Lemma (\cite[Theorem 3]{brustle2}). In part, it shows that if a maximal green sequence for a quiver $Q$ first mutates vertex $k$, then the quiver obtained from mutating $Q$ at $k$ also has a maximal green sequence. Repeated application of this result then shows that any intermediate quiver in a maximal green sequence has a maximal green sequence. The Rotation Lemma gives the existence of a maximal green sequence for many quivers in a mutation class, but does not prove Theorem \ref{thm:main}.

A quiver $Q$ is said to be of finite mutation type if $\mut(Q)$ is finite.
It is known that all finite mutation type quivers arise from triangulations of surfaces except for the rank 2 case and 11 exceptional cases \cite{felikson}. These 11 cases are $\mathbb{E}_6,\mathbb{E}_7,\mathbb{E}_8,\widetilde{\mathbb{E}_6},\widetilde{\mathbb{E}_7},\widetilde{\mathbb{E}_8},\mathbb{E}_6^{(1,1)},\mathbb{E}_7^{(1,1)},\mathbb{E}_8^{(1,1)}, \mathbb{X}_6,$ and $\mathbb{X}_7.$ Among these exceptional cases it has been shown that there exists a quiver with a maximal green sequence for all but $\mathbb{X}_7$ \cite{alim}. It was shown in \cite{seven} that neither of the two quivers in the mutation class of $\mathbb{X}_7$ have a maximal green sequence. All rank 2 cluster algebras have a maximal green sequence given by first mutating at the source vertex and then mutating at the other vertex. 

We use the cluster algebra package in the computer program Sage to produce an explicit maximal green sequence for every quiver in the outstanding exceptional cases to obtain our second result.

\begin{theorem}[Theorem \ref{thm:main2}]\label{thm:main2f}
If $Q$ is a quiver in the mutation class of $\widetilde{\mathbb{E}_6},\widetilde{\mathbb{E}_7},\widetilde{\mathbb{E}_8},\mathbb{E}_6^{(1,1)},\mathbb{E}_7^{(1,1)},$ $\mathbb{E}_8^{(1,1)},$ or $\mathbb X_6$, then $Q$ has a maximal green sequence. 
\end{theorem}

It is already known that every quiver in the mutation class of $\mathbb{E}_6,\mathbb{E}_7,$ and $\mathbb{E}_8$ have maximal green sequences by \cite{brustle}. By combining previous results with Theorem \ref{thm:main} and Theorem \ref{thm:main2f} we have a complete classification of which quivers of finite mutation have a maximal green sequence. Furthermore, we have shown that the existence of a maximal green sequence is mutation-invariant for quivers of finite mutation type. 

\begin{theorem}[Theorem \ref{thm:main3}]\label{con:muteq}
Let $Q$ be a quiver of finite mutation type, then a maximal green sequence exists for every quiver in $\mut(Q),$ or there is no maximal green sequence for any quiver in $\mut(Q)$.
In particular, $Q$ has a maximal green sequence unless it arises from a triangulation of a once-punctured closed surface, or is one of the two quivers in the mutation class of $\mathbb{X}_7$. 
\end{theorem}

The existence of a maximal green sequence for a quiver also seems to be related to whether the cluster algebra $\mathcal{A}$ it generates is equal to its upper cluster algebra $\mathcal{U}$. Gross, Hacking, Keel and Kontsevich showed that if $\mathcal{A}= \mathcal{U}$ and a maximal green sequence exists, then the Fock-Goncharov canonical basis conjecture holds \cite{gross}. It is still unknown as to whether or not $\mathcal{A}=\mathcal{U}$ for closed higher genus surfaces with at least two punctures and punctured closed spheres. For all other quivers of finite mutation type it is known that $\mathcal{A}=\mathcal{U}$ if and only if there exists a quiver with a maximal green sequence. See \cite{lee} and references therein for more information on when $\mathcal{A}=\mathcal{U}.$ 

In Section \ref{sec:quivers} we give background on quivers and maximal green sequences. In Section \ref{sec:surfaces} we give background on marked surfaces and their triangulations. In Section \ref{sec:cycleindep} we discuss two mutation sequences that are used in Section \ref{sec:proof}, where we give the construction for maximal green sequences for closed surfaces. In Section \ref{sec:boundary} we prove the existence of maximal green sequences for surfaces with boundary. We then discuss the maximal green sequences for exceptional cases in Section \ref{sec:exceptional}. 

\section{Quivers and maximal green sequences}\label{sec:quivers}

We recall the definitions from \cite{keller2}, but use the conventions given in \cite{brustle}.

\begin{definition} A \textbf{(cluster) quiver} is a directed graph with no loops or 2-cycles. An \textbf{ice quiver} is a pair $(Q, F)$ where $Q$ is a quiver and $F$ is a subset of the vertices of $Q$ called \textbf{frozen vertices;} such that there are no edges between frozen vertices. If a vertex of $Q$ is not frozen it is called \textbf{mutable}. For convenience, we assume that the mutable vertices are labelled $\{1,\ldots, n\}$, and frozen vertices are labeld by $\{n+1, \ldots, n+m\}$. 
\end{definition}

\begin{definition}\label{def:mutation}
Let $(Q,F)$ be an ice quiver, and $k$ a mutable vertex of $Q$. The \textbf{mutation} of $(Q,F)$ at vertex $k$ is denoted by $\mu_k$, and is a transformation $(Q,F)$ to a new ice quiver $(\mu_k(Q),F)$ that has the same vertices, but making the following adjustment to the edges:  \begin{enumerate}
\item For every 2-path $i \rightarrow k \rightarrow j$, add a new arrow $i \rightarrow j$. 
\item Reverse the direction of all arrows incident to $k$. 
\item Delete any 2-cycles created during the first two steps, and any arrows between frozen vertices. 
\end{enumerate} \end{definition} 
Mutation at a vertex is an involution, and an equivalence relation. We define $\mut(Q)$ to be the equivalence class of all quivers that can be obtained from $Q$ by a sequence of mutations.

\begin{definition}
 
Let $Q_0$ be the set of vertices of $Q$. The \textbf{framed quiver} associated with a quiver $Q$ is the ice quiver $(\hat{Q},Q_0')$ such that:

$$Q_0'=\{i'\text{ }|\text{ }i\in Q_0\}, \hspace{.6cm} \hat{Q}_0 = Q_0 \sqcup Q_0'$$
$$\hat{Q}_1 = Q_1 \sqcup \{i \to i'\text{ }|\text{ }i \in Q_0\}$$
\end{definition}

Since the frozen vertices of the framed quiver are so natural we will simplify the notation and just write $\hat{Q}$. Now we must discuss what is meant by red and green vertices.

\begin{definition}\label{def:green1}
Let $R \in \mut(\hat{Q})$. \\A mutable vertex $i \in R_0$ is called \textbf{green} if $$\{j'\in Q_0'\text{ }| \text{ } \exists \text{ } j' \rightarrow i \in R_1 \}=\emptyset.$$ It is called \textbf{red} if $$\{j'\in Q_0'\text{ }| \text{ } \exists \text{ } j' \leftarrow i \in R_1 \}=\emptyset.$$ 
\end{definition}

It is not clear from the definition that every mutable vertex in $R_0$ is either red or green. In the case of quivers this result is due to \cite{derksen} and then it was also shown in a more general setting in \cite{gross}.

\begin{theorem}\cite{derksen,gross}\label{thm:signcoh}
Let $R \in \mut(\hat{Q})$. Then every mutable vertex in $R_0$ is either red or green.
\end{theorem}

\begin{definition}
A \textbf{green sequence} for $Q$ is a sequence $\textbf{i}=(i_1, \dots, i_l) \subset Q_0$ such that $i_1$ is green in $\hat{Q}$ and for any $2\leq k \leq l$, the vertex $i_k$ is green in $\mu_{i_{k-1}}\circ \cdots \circ \mu_{i_1}(\hat{Q})$. 
%The integer $l$ is called the length of the sequence $\textbf{i}$ and is denoted by $l(\textbf{i})$. 
A green sequence \textbf{i} is called maximal if every mutable vertex in $\mu_{i_{l}}\circ \cdots \circ \mu_{i_1}(\hat{Q})$ is red.
\end{definition}

\section{Marked surfaces and their triangulations}\label{sec:surfaces}

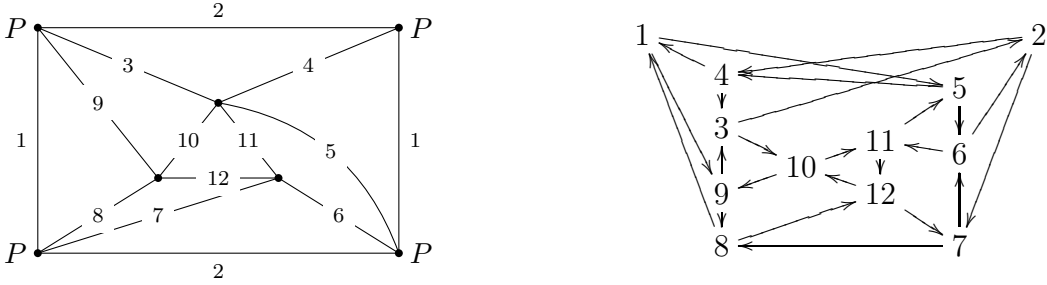
\begin{figure}[t]
\centering
\begin{minipage}{.5\textwidth}
\begin{tikzpicture}[x=.8cm,y=.5cm,font=\tiny]
\coordinate (a) at (0,0);
\coordinate (b) at (6,0);
\coordinate (c) at (6,6);
\coordinate (d) at (0,6);
\coordinate (g) at (2,2);
\coordinate (f) at (4,2);
\coordinate (e) at (3,4);
\fill (a) circle (1.5pt);
\fill (b) circle (1.5pt);
\fill (c) circle (1.5pt);
\fill (d) circle (1.5pt);
\fill (e) circle (1.5pt);
\fill (f) circle (1.5pt);
\fill (g) circle (1.5pt);

\draw (a) node[left]{{\small $P$}};
\draw (d) node[left]{{\small $P$}};
\draw (b) node[right]{{\small $P$}};
\draw (c) node[right]{{\small $P$}};

\draw (a) to node [midway, below] {\tiny{2}} (b);
\draw (b) to node [midway, right] {\tiny{1}} (c);
\draw (c) to node [midway, above] {\tiny{2}} (d);
\draw (d) to node [midway, left] {\tiny{1}} (a);

\draw (d) to node [midway,fill=white]{3} (e);
\draw (e) to node [midway,fill=white]{4} (c);
\draw (e) to[bend left] node [midway,fill=white]{5} (b);
\draw (f) to node [midway,fill=white]{6} (b);
\draw (f) to node [midway,fill=white]{7} (a);
\draw (g) to node [midway,fill=white]{8} (a);
\draw (g) to node [midway,fill=white]{9} (d);
\draw (g) to node [midway,fill=white]{10} (e);
\draw (e) to node [midway,fill=white]{11} (f);
\draw (g) to node [midway,fill=white]{12} (f);
\end{tikzpicture}
\end{minipage}
\begin{minipage}{.5\textwidth}
\begin{xy} 0;<.3pt,0pt>:<0pt,-.2pt>:: 
(0,0) *+{1} ="0",
(500,0) *+{2} ="1",
(100,175) *+{3} ="2",
(100,75) *+{4} ="3",
(400,100) *+{5} ="4",
(400,225) *+{6} ="5",
(400,400) *+{7} ="6",
(100,400) *+{8} ="7",
(100,300) *+{9} ="8",
(200,250) *+{10} ="9",
(300,200) *+{11} ="10",
(300,300) *+{12} ="11",
"3", {\ar"0"},
"0", {\ar"4"},
"7", {\ar"0"},
"0", {\ar"8"},
"2", {\ar"1"},
"1", {\ar"3"},
"5", {\ar"1"},
"1", {\ar"6"},
"3", {\ar"2"},
"8", {\ar"2"},
"2", {\ar"9"},
"4", {\ar"3"},
"4", {\ar"5"},
"10", {\ar"4"},
"6", {\ar"5"},
"5", {\ar"10"},
"6", {\ar"7"},
"11", {\ar"6"},
"8", {\ar"7"},
"7", {\ar"11"},
"9", {\ar"8"},
"9", {\ar"10"},
"11", {\ar"9"},
"10", {\ar"11"},
\end{xy}
\end{minipage}
\caption{A triangulation $T_1$ of $\Sigma_1=\{1,0,4,\emptyset\}$ (left), and the corresponding quiver $Q_{T_1}$ (right).}\label{ex:1}
\end{figure}

To begin the section we recall the definition of a marked surface given in \cite{fomin}.
Let $S$ be an orientable 2-dimensional Riemann surface with or without boundary. We designate a finite number of points, $M$, in the closure of $S$ as marked points. We require at least one marked point on each boundary component. We call marked points in the interior of $S$ \textbf{punctures.} Together the pair $\Sigma=(S,M)$ is called a \textbf{marked surface.} 
For technical reasons we exclude the cases when $\Sigma$ is one of the following: 
\begin{itemize}
\item a sphere with less than four punctures;
\item an unpunctured or once punctured monogon;
\item an unpunctured digon; or 
\item an unpunctured triangle. 
\end{itemize}
Note that the construction allows for spheres with four or more punctures.

Up to homeomorphism a marked surface is determined by four things. The first is the genus $g$ of the surface. The second is the number of boundary components $b$ of $S$. The third is the number of punctures $p$ in $M$, and the fourth is the set $m=\{m_i\}_{i=1}^{b}$ where $m_i\in \mathbb{Z}_{>0}$ denotes the number of marked points on the $i$th boundary component of $S.$ We say a marked surface is \textbf{closed} if it has no boundary. 
\begin{definition}
An \textbf{arc} $\gamma$ in $(S,M)$ is a curve in $S$ such that:
\begin{itemize}
\item The endpoints of $\gamma$ are in $M$. 
\item $\gamma$ does not intersect itself, except that its endpoints may coincide.
\item $\gamma$ is disjoint from $M$ and the boundary of $S$, except at its endpoints. 
\item $\gamma$ is not isotopic to the boundary, or the identity. 
\end{itemize}
An arc is called a \textbf{loop} if its two endpoints coincide.
% We differentiate between two types of loops, namely separating and non-separating loops. A loop $\gamma \subset T$ is \textbf{separating} if $T \setminus \gamma$ consists of two connected components, and \textbf{non-separating} otherwise. We will refer to a separating loop as a monogon. 

Each arc is considered up to isotopy. Two arcs are called compatible if there exists two arcs in their respective isotopy classes that do not intersect in the interior of $S$. 
\end{definition}

\begin{definition}
A \textbf{taggd arc} is constructed by taking an arc that does not cut out a once-punctured monogon and marking or "tagging" its ends as either \textbf{plain} or \textbf{notched} so that: \begin{itemize}
\item an endpoint lying on the boundary of $S$ is tagged plain; and
\item both ends of a loop must be tagged in the same way. 
\end{itemize} 
We use a $\bowtie$ to denote the tagging of an arc in figures. Two tagged arcs are considered compatible if: \begin{itemize}
\item Their underlying untagged arcs are the same, and their tagging agrees on exactly one endpoint. 
\item Their underlying untagged arcs are distinct and compatible, and any shared endpoints have the same tagging. 
\end{itemize}
 A maximal collection of pairwise compatible tagged arcs is called a \textbf{(tagged) triangulation} of $(S,M)$. 
\end{definition}

\begin{figure}[t]
\begin{tikzpicture}
\draw (0,0) to[bend left] node[above]{$\ell$} (9,0) to[bend left] node [below]{$i$} (0,0);
\draw (0,0) to[bend left]  node [below]{$k$} node[near end,sloped,rotate=90]{$\bowtie$} (5,0) to[bend left] node [above]{$j$} (0,0);

\fill (0,0) circle (1.5pt);
\fill (5,0) circle (1.5pt);
\fill (9,0) circle (1.5pt);
\draw (5,0) node[right] {$P$};
\end{tikzpicture}
\caption{The puncture $P$ is a radial puncture.}\label{fig:radial}
\end{figure}
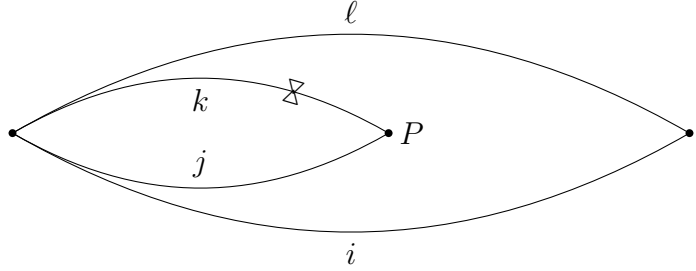

\begin{definition}
We call a puncture $P$ a \textbf{radial puncture in a tagged triangulation} if and only if $P$ is the unique puncture in the interior of a digon and there exists two arcs in the interior of this digon that differ only by their tagging at $P$. See Figure \ref{fig:radial}.
\end{definition}

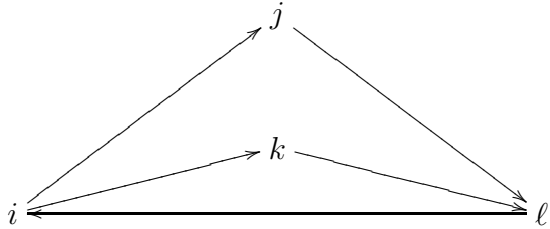
\begin{figure}[t]

\centering
$$
\begin{xy} 0;<1pt,0pt>:<0pt,-1pt>:: 
(0,75) *+{i} ="0",
(100,0) *+{j} ="1",
(100,50) *+{k} ="2",
(200,75) *+{\ell} ="3",
"0", {\ar"1"},
"0", {\ar"2"},
"1", {\ar"3"},
"2", {\ar"3"},
"3",{\ar"0"},
\end{xy}$$
\caption{A special situation in the construction of a quiver from the triangulation given in Figurex\ref{fig:radial}.}\label{fig:arc}
\end{figure}

\begin{definition}\label{def:quivfromtri}
Let $T$ be a triangulation of a marked surface. The \textbf{quiver associated to $T$}, which we will denote as $Q_T$, is the quiver obtained from the following construction. For each arc $\alpha$ in a triangulation $T$ add a vertex $v_\alpha$ to $Q_T$. If $\alpha_i$ and $\alpha_j$ are two edges of a triangle in $T$ with $\alpha_j$ following $\alpha_i$ in a clockwise order, then add an edge to $Q_T$ from $v_{\alpha_i} \rightarrow v_{\alpha_j}$. If $\alpha_k$ and $\alpha_j$ have the same underlying untagged arc as in Figure \ref{fig:radial} we refer you to Figure \ref{fig:arc} for the construction in this situation. Note that the quiver is the same whether $\alpha_j$ or $\alpha_k$ is tagged. More generally distinct triangulations may yield the same quiver. 
\end{definition}

We now define the analog of quiver mutation for triangulations of a marked surface. 

\begin{definition}
A \textbf{flip} is a transformation of a triangulation that removes an arc $\gamma$ and replaces it with a (unique) different arc $\gamma'$ that, together with the remaining arcs, forms a new triangulation $T'$. In this case we define $\mu_\gamma(T)=T'$. This makes sense by the following lemma. 
\end{definition}

\begin{lemma}\cite[Lemma 9.7]{fomin}\label{lem:mut=flip}
Let $T$ and $T'$ be two triangulations related by a flip of an arc $\gamma$. Suppose $\gamma$ corresponds to vertex $k$ of $Q_T$, then $Q_{T'} = \mu_k(Q_T)$.
\end{lemma}
In Sections \ref{sec:cycleindep} and \ref{sec:proof} we will exclusively refer to a flip of an arc in a triangulation as a mutation. 

Thurston's theory of laminations and shear coordinates provide a way to introduce frozen vertices in the geometric setting. 

\begin{definition}\cite[Definition 12.1]{ft}
A \textbf{lamination} on a marked surface $(S,M)$ is a finite collection of non-self-intersecting and pairwise non-intersecting curves in $S$ up to isotopy. Each curve must be one of the following: \begin{itemize}
\item a closed curve;
\item a curve connecting two unmarked points on the boundary of $S$;
\item a curve starting at an unmarked point on the boundary, and at its other end spiraling into a puncture (either clockwise or counter clockwise);
\item a curve whose ends both spiral into punctures (not necessarily distinct). 
\end{itemize} 
We forbid any curves that bound an unpunctured or once-punctured disk, curves with two endpoints on the boundary which are isotopic to a piece of boundary containing zero or one marked points, and a curve with two ends spiraling into the same puncture in the same direction without enclosing anything else. 
\end{definition}

\begin{figure}
\begin{tikzpicture}[x=2cm,y=2cm]
\draw[ultra thick] (0,0) to (1,0);
\draw (1,0) to (1,1);
\draw[ultra thick] (1,1) to (0,1);
\draw (0,1) to (0,0);
\draw[very thick] (0,1) to node[near start,below]{$\alpha$} (1,0);

\draw[ultra thick] (3,0) to (4,0);
\draw (4,0) to (4,1);
\draw[ultra thick] (4,1) to (3,1);
\draw (3,1) to (3,0);

\draw[ultra thick] (3,0) to node[near end,below]{$\alpha$}(4,1);
\draw (.5,1.25) node[above]{$\ell$} to[out=260, in= 80,looseness=1] (.5,-0.25);
\draw (3.5,1.25) node[above]{$\ell$} to[out=260, in= 80,looseness=1] (3.5,-.25);
\end{tikzpicture}
\caption{On the left, the curve $\ell \in L$ contributes a +1 to the shear coordinate for $\alpha$. On the right, $\ell$ contributes a -1. We have made the arcs that intersect $\ell$ bold to emphasize the ``$S$'' and ``$Z$'' shapes.}\label{fig:shearcoord}
\end{figure}
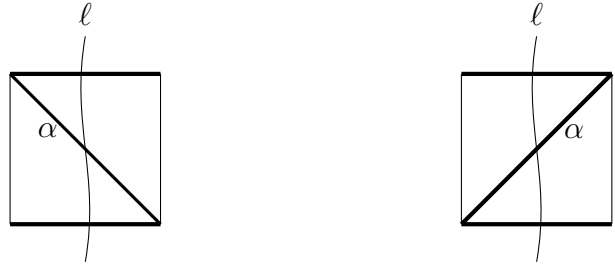
\begin{definition}\cite[Definition 12.2,13.1]{ft}\label{def:lamination}
Let $L$ be a lamination, and let $T$ be a triangulation without any arcs that are notched. Note that this requires that there are no radial punctures in $T$. For each $\alpha \in T$, the corresponding \textbf{shear coordinate} of $L$ with respect to $T$ denoted $b_\alpha(T,L),$ is defined as a sum of contributions from all intersections of curves in $L$ with $\alpha$. An intersection contributes a +1 (resp., -1) to $b_\alpha(T,L)$ if the corresponding segment of a curve in $L$ cuts through the quadrilateral surrounding $\alpha$ cutting through the edges in the shape of an ``$S$'' (resp., in the shape of a $``Z$''), as in Figure \ref{fig:shearcoord}. See Figure \ref{fig:shearcoord}.
Note that at most one of these two types of intersections can occur and $b_\alpha(T,L)$ is always finite. 

For a triangulation $T$ that contains notched arcs the shear coordinates are uniquely defined by the following rules:
\begin{enumerate}
\item Suppose that tagged triangulations $T_1$ and $T_2$ coincide except that at a particular puncture $P$, the tags of the arcs in $T_1$ are all different from the tags of their counterparts in $T_2$. Suppose that laminations $L_1$ and $L_2$ coincide except that each curve in $L_1$ that spirals into $P$ has been replaced in $L_2$ by a curve that spirals in the opposite direction. Then $b_{\alpha_1}(T_1,L_1)=b_{\alpha_2}(T_2,L_2)$ for each tagged arc $\alpha_1 \in T_1$ and its counterpart $\alpha_2 \in T_2$. 

\item By performing tag-changing transformations $L_1 \rightarrow L_2$ with $L_1$ and $L_2$ as above, we can convert any triangulation into a triangulation $T$ that does not contain any notches except possibly at radial punctures. If $\alpha \in T$ is not notched at any puncture, then we define $b_\alpha(T,L)$ as above for the underlying plain arc.
\end{enumerate}
 If $\alpha \in T$ is the arc incident to a radial puncture with different taggings at its endpoints, then we apply the tag-changing transformation in rule (1) to the radial puncture and then use rule (2) to compute $b_\alpha(T,L)$.
\end{definition}
Note that the quiver constructed from a triangulation in Definition \ref{def:quivfromtri} does not include any frozen vertices. We can use laminations and shear coordinates to extend Theorem \ref{lem:mut=flip} to show that the flips in triangulations of surfaces together with laminations agree with mutations of ice quivers. 
\begin{definition}
 A \textbf{multi-lamination} is a finite family of laminations. Let $T$ be a triangulation for a marked surface and $\mathcal{L}$ a multi-lamination. Let $Q_T$ be the quiver constructed in Definition \ref{def:quivfromtri}. Suppose that the arcs in $T=\{\alpha_i\}_{i=1}^n$ are indexed by their corresponding vertex of $Q_T$ and that $\mathcal{L}=\{L_j\}_{j=n+1}^m$. We define an ice quiver $(\widetilde{Q_T},F_\mathcal{L})$ where $$F_\mathcal{L}=\{j | L_j \in \mathcal{L}\}, (\widetilde{Q_T})_0=(Q_T)_0 \sqcup F_\mathcal{L},$$ and the edges of $\widetilde{Q_T}$ are the edges of $Q_T$ together with $b_{\alpha_i}(T,L_j)$ edges $i \rightarrow j$ for all $i=1,\dots,n$ and $j=n+1,\dots,m$. Note that a negative value for $b_{\alpha_i}(T,L_j)$ corresponds to adding $|b_{\alpha_i}(T,L_j)|$ edgess $j \rightarrow i$. 
\end{definition}

\begin{lemma} \cite[Theorem 13.5]{ft}\label{lem:flip=mut2} 
Let $T$ and $T'$ be two triangulations related by a flip of an arc $\gamma$. Suppose $\gamma$ corresponds to vertex $k$ of $\widetilde{Q_T}$, then $\widetilde{Q_{T'}} = \mu_k(\widetilde{Q_T})$.
\end{lemma}

%Suppose that $T_1,T_2,L_1$ and $L_2$ are as above in (1) with arcs being tagged plain in $T_1$ and notched in $T_2$ at puncture $P$.
%In the following sections it will be necessary to calculate $b_\alpha(T_2,L_1)$ and it will be convenient to be able to calculate shear coordinates of notched arcs without having to constantly apply tag-changing transformations. To this affect if $\alpha \in T_2$ we define $b'_\alpha(T_2,L_2)$ be the shear coordinate of the untagged arc underlying $\alpha.$ 
%
%\begin{lemma}\label{lem:altshear} 
%For every arc $\alpha \in T_2$ we have $$b_\alpha(T_2,L_1)=b'_\alpha(T_2,L_2).$$
%\end{lemma}
%\begin{proof}
%By rule (1) we have $b_\alpha(T_2,L_1)=b_{\alpha_1}(T_1,L_2)$ where $\alpha_1$ is the counterpart to $\alpha$ in $T_1$, but $b_{\alpha_1}(T_1,L_2)=b'_\alpha(T_2,L_2)$ since the underlying untagged arcs of $\alpha$ and $\alpha_1$ are identical. 
%\end{proof}
%
%Therefore when it is necessary to calculate $b_\alpha(T_2,L_1)$ when we are working with maximal green sequences we will not need to perform any tag-changing transformations. We will just need to calculate the shear coordinate of the underlying untagged arc with a different lamination. 

\begin{definition}\label{def:elelam}
Let $\alpha$ be a tagged arc of a marked surface. Denote by $L_\alpha$ a lamination consisting of a single curve defined as follows. The curve $L_\alpha$ runs along $\alpha$ within a small neighborhood of it. If $\alpha$ has an endpoint $a$ on a component $C$ of the boundary of $S$, then $L_\alpha$ begins at a point $a' \in C$ located near $a $ in the clockwise direction, and proceeds along $\alpha$. If $\alpha$ has an endpoint at a puncture, then $L_\alpha$ spirals into $a$: clockwise if $\alpha$ is tagged plain at $a$, and counterclockwise if it is notched. We call $L_\alpha$ the \textbf{elementary lamination associated to $\alpha$}. 
\end{definition}

Note that for any arc $\alpha$ in a triangulation $T,$ $L_\alpha$ is the unique lamination such that $$b_\gamma(T,L_\alpha) = \begin{cases} -1 & \text{if } \gamma = \alpha, \\ 0 &\text{if } \gamma \neq \alpha.\end{cases}$$ If we fix the multi-lamination $\mathcal{L}= \{L_\alpha | \alpha \in T\}$ then the ice quiver $\widetilde{Q_T}$ is identical to the framed quiver $\widehat{Q_T}.$

\begin{remark}
The elementary laminations defined in Definition \ref{def:elelam} are not the same elementary laminations given in \cite{ft}. There $L_\alpha$ is the unique lamination that contributes a +1 to $b_\alpha(T,L)$ and 0 for all other arcs.
\end{remark}
We now give the geometric characterization for what it means for an arc to be green or red. 

\begin{definition}\label{def:green2}
Let $T$ be a triangulation of a marked surface. Fix the multi-lamination $\mathcal{L}=\{L_\alpha\}_{\alpha\in T}$ where $L_\alpha$ is the elementary lamination associated to $\alpha$. Let $T'$ be a triangulation obtained from $T$ by some finite sequence of flips. Then $\alpha' \in T'$ is said to be \textbf{green} if $$\{L \in L^\circ | b_{\alpha'}(T',L) > 0\} = \emptyset.$$ It is called \textbf{red} if $$\{L \in L^\circ | b_{\alpha'}(T',L) < 0\} = \emptyset.$$
\end{definition}

It follows from Lemma \ref{lem:flip=mut2} that an arc $\alpha' \in T'$ is green (resp., red) in the sense of Definition \ref{def:green2} if and only if its corresponding vertex in $(\widetilde{Q_{T'}},F_\mathcal{L})$ is green (resp., red) in the sense of Definition \ref{def:green1}. 

%In the following sections we will use laminations to prove that the sequences provided are indeed maximal green sequences. Note that by sign-coherency to show that a particular vertex/arc is green it suffices to show that there exists a single lamination that contributes a negative shear coordinate. We implicitly use this fact in many of the following proofs.

We also provide one more lemma about green sequences that we will use in the sequel. 

\begin{lemma}\label{lem:vertexdone}
Let $Q$ be a quiver and $i$ a vertex in $Q$. If there exists a frozen vertex $j'$ such that there is exactly one arrow incident to $j'$ and it points at $i$, then $i$ is red and will never be mutated at in any green sequence for $Q$. 
\end{lemma}
\begin{proof}
Since $j'$ is a frozen vertex pointing at $i$, we have that $i$ is not green and by Theorem \ref{thm:signcoh} it must be red and therefore cannot be mutated at in a green sequence. Mutating at any vertex other than $i$ in the quiver will not affect the edge $j' \rightarrow i$, so this edge will persist through any mutation sequence and $i$ will always be red. 
\end{proof}

Translating this lemma into the language of laminations, it says that if there exists an arc $\alpha$ and a lamination $L$ such that $$b'_{\gamma}(T,L)=\begin{cases} 1 & \gamma=\alpha, \\ 0 & \gamma \neq \alpha.\end{cases}$$ then $\alpha$ will be red in any green mutation sequence for $T$.

\section{Cycle Lemma and Independent mutation sequence}\label{sec:cycleindep}

In any triangulation with at least three distinct arcs (none of which are loops) incident to a puncture, these arcs form an oriented cycle in the corresponding quiver. The mutation sequences given in this paper make use of the following maximal green sequence for oriented cycles.

\begin{definition}
Let $i$ and $j$ be vertices of a quiver $Q$. Suppose $\mu$ is some mutation sequence for $Q$. Let $\overline{\mu(Q)}$ denote the quiver $\mu(Q)$ with the relabelling of the vertices that fixes the label on vertex $k$ if $k \neq i,j;$ but relabels $i$ as $j$ and vice versa. 
We say that $\mu$ \textbf{interchanges $i$ and $j$} if $Q = \overline{\mu(Q)}$. 
\end{definition}
\begin{lemma}[Cycle Lemma]{\cite[Lemma 4.2]{bucher}} \label{cycle}
Let $C$ be a quiver that is an oriented $n$-cycle with vertices labeled $1,\ldots,n$, and edges $i \rightarrow (i-1)$ for $ 2 \leq i \leq n$ and $1 \rightarrow n$. Define the mutation sequence $$\mucycle:=(n,{n-1}\ldots ,2,1, 3 ,4, \ldots, ,{n-1},n).$$ Then $\mucycle$ is a maximal green sequence for $C$ that interchanges 1 and 2. 
\end{lemma}

It is helpful in the next section to understand how the Cycle Lemma affects triangulations. 

\begin{lemma}\label{lem:cyclesurface}
Let $T$ be a triangulation of a marked surface. Suppose that a puncture $P$ is incident to at least three distinct arcs, none of which are loops, so these arcs correspond to an oriented cycle in $Q_T.$ Let $\mucycle^P$ be the mutation sequence from Lemma \ref{cycle} for this oriented cycle. Then $\mucycle^P(T)$ coincides with $T$ (up to relabelling of arcs) except all of the taggings of the arcs at $P$ differ. 
\end{lemma}
\begin{proof}
Let $P$ and $T$ be as above. Without loss of generality we assume that all of the arcs incident to $P$ are tagged plain. Suppose $\mucycle^P = (\alpha_n, \dots, \alpha_2,\alpha_1, \alpha_3, \dots, \alpha_n).$ Let $\lambda = (\alpha_n , \dots , \alpha_3)$ be the first part of $\mucycle^P$. Note that in $\mu_{\alpha_2}\lambda(T)$ we have that $P$ is a radial puncture incident to $\alpha_1$ and $\alpha_2$, with $\alpha_2$ notched at $P$. Now when we mutate $\alpha_1$ it will again be incident to $P$ and must be tagged at $P$ by the compatibility rules for tagged arcs. Furthermore, the triangulation $\mu_{\alpha_1}\mu_{\alpha_2}\lambda(T)$ is identical to $\lambda(T)$ except that $\alpha_1$ and $\alpha_2$ are now notched at $P$. It follows then that the rest of $\mucycle^P$ mutates all the other arcs back into place, but by the compatibility rules for tagged arcs they must all be notched at $P$.
\end{proof}
%\begin{corollary}
In the degenerate case when there are exactly two arcs incident to a puncture $P$, the corresponding vertices in the quiver are not connected by an edge. If $P$ is not a radial puncture then we define $\mucycle^P:=(1,2)$ and note that the result of Lemma \ref{lem:cyclesurface} still applies in this case. 
%\end{corollary}

\subsection{Independent mutation sequence}
Br\'ustle and Qiu showed that it is necessary that a maximal green sequence for a triangulation of a surface with punctures must change the tagging at every puncture \cite{brustleqiu}. One might hope that we could apply $\mucycle$ to every puncture of a closed marked surface to obtain a maximal green sequence, but that is not the case. 

The two issues with this approach are when there are two distinct punctures $P$ and $R$ with an arc between them and loops at punctures. The mutation sequence $\mucycle^R\mucycle^P$ would not be a green sequence as the arc between $P$ and $R$ will be red in $\mucycle^P(T)$. 
These problems motivate the following mutation sequence. This sequence mutates us to a triangulation where every puncture in a designated proper subset of punctures will not be the base point of any loops, and will not share an arc with any other puncture in the subset. Then we may proceed to apply the Cycle Lemma to each oriented cycle around each puncture in this subset. 

\begin{definition}\label{def:independent}
Let $P$ and $R$ be two not necessarily distinct punctures of a closed surface with at least two punctures. We say that $P$ is \textbf{independent} of $R$ in $T$ if there is no arc in $T$ with one endpoint at $P$ and the other at $R$. A set $\mathcal{P}$ of marked points is called \textbf{independent} in $T$ if for any two not necessarily distinct points $P,R \in \mathcal{P}$ we have $P$ is independent of $R$.
\end{definition}

\begin{definition}\label{def:indpath} Let $\alpha$ be an arc. Let $\iota(\alpha)$ denote the underlying untagged arc of $\alpha$ if $\alpha$ is not incident to a radial puncture, or incident to a radial puncture with both endpoints tagged the same way. If $\alpha$ is incident to a radial puncture and its endpoints have different taggings let $\iota(\alpha)$ be a loop enclosing the radial puncture based at the other endpoint of $\alpha$. 

Let $T$ be a triangulation of a closed marked surface and $\mathcal{P}$ a proper subset of its punctures.
Define $$E_T^\mathcal{P}:=\{\alpha \in T | \iota(\alpha) \text{ has two endpoints in } \mathcal{P}\}.$$
An \textbf{independence path} for $\alpha \in E_T^\mathcal{P}$ is a path from some point $x_\alpha \in \alpha$ to some puncture not in $\mathcal{P}$, such that the path is disjoint from punctures of $\mathcal{P}$ and disjoint from any arcs not contained in $E_T^\mathcal{P}$. 

\begin{example}\label{ex:indpath}
Consider the triangulation given in Figure \ref{fig:indpath}. Take $\mathcal{P}=\{P_i\}_{i=1}^8$. The set $$E_T^\mathcal{P}=\{3,4,5,6,7,8,9,10,11,12,13,14,15,16,17,18,\alpha,\beta\}.$$ An independence path for the arc labelled $\alpha$ is shown in the figure in green. Note that the arc $\beta \in E_T^\mathcal{P}$ since $\iota(\beta)$ is a loop based at $P_1$ enclosing the puncture $S$. We draw $\iota(\beta)$ on the triangulation in blue. 
\end{example}
\begin{figure}[t]
\begin{tikzpicture}[font=\tiny]
\coordinate (N) at (7,4);
\coordinate (S) at (-5,-4);
\coordinate (E) at (7,-4);
\coordinate (W) at (-5,4);
\fill (N) node[right]{$P_0$} circle (1.5pt);
\fill (S) node[left]{$P_0$} circle (1.5pt);
\fill (E) node[right]{$P_0$} circle (1.5pt);
\fill (W) node[left]{$P_0$} circle (1.5pt);

\coordinate (p1) at (3,0);
\coordinate (p2) at (0,3);
\coordinate (p3) at (-3,0);
\coordinate (p4) at (0,-3);
\coordinate (p5) at (1,1);
\coordinate (p6) at (-1,1);
\coordinate (p7) at (-1,-1);
\coordinate (p8) at (1,-1);
\coordinate (p9) at (5,0);

\fill (p1) node[above]{$P_1$} circle (1.5pt);
\fill (p2) node[above]{$P_2$} circle (1.5pt);
\fill (p3) node[left]{$P_3$} circle (1.5pt);
\fill (p4) node[below]{$P_4$} circle (1.5pt);
\fill (p5) node[above right]{$P_5$} circle (1.5pt);
\fill (p6) node[above left]{$P_6$} circle (1.5pt);
\fill (p7) node[below left]{$P_7$} circle (1.5pt);
\fill (p8) node[below right]{$P_8$} circle (1.5pt);
\fill (p9) circle (1.5pt);

\draw (N) to node[fill=white]{1} (E) to node[fill=white]{2} (S) to node[fill=white]{1} (W) to node[fill=white]{2} (N);
\draw (p1) to node[above right]{3} (p2) to node[above left]{4} (p3) to node[below left]{5} (p4) to node[below right]{6} (p1);
\draw (p5) to node[fill=white]{7} (p6) to node[fill=white]{8} (p7) to node[fill=white]{9} (p8) to node[fill=white]{10} (p5);
\draw (p1) to node[fill=white]{11} (p5) to node[fill=white]{12} (p2) to node[fill=white]{13} (p6) to node[fill=white]{14} (p3) to node[fill=white]{15} (p7) to node[fill=white]{16} (p4) to node[fill=white]{17} (p8) to node[fill=white]{18} (p1);
\draw (p5) to node[near start,fill=white]{$\alpha$}(p7);
\draw (N) to node[fill=white]{19} (p2) to node[fill=white]{20} (W) to node[fill=white]{21} (p3) to node[fill=white]{22} (S) to node[fill=white]{23} (p4) to node[fill=white]{24} (E) to node[fill=white]{25} (p1) to node[fill=white]{26} (N);
\draw (p1)[bend left] to node[midway,fill=white] {$\beta$} node[near end,sloped, rotate=90]{$\bowtie$} (p9) node[right]{$S$} to[bend left] node[fill=white]{27} (p1);
\draw (p1) to[out=-40,in=180] (5.3,-.8) to[out=0,in=240] node[fill=white]{28} (N);
\draw[dashed,color = blue, very thick] (p1) to[out=40,in=90] (5.5,0) to[out=-90,in=-35] (p1);
\draw[color=blue] (5,1) node {$\iota(\beta)$};
\draw[dashed, color= darksagegreen, very thick] (0,0) to[out=90,in=-45] (W);
\fill[color= darksagegreen] (0,0) node[below right]{$x_\alpha$}circle (2pt);
\end{tikzpicture}
\caption{A triangulation of a torus with 10 punctures. An independence path for the arc labelled $\alpha$ is shown in green, and $\iota(\beta)$ is shown in blue.}\label{fig:indpath}
\end{figure}
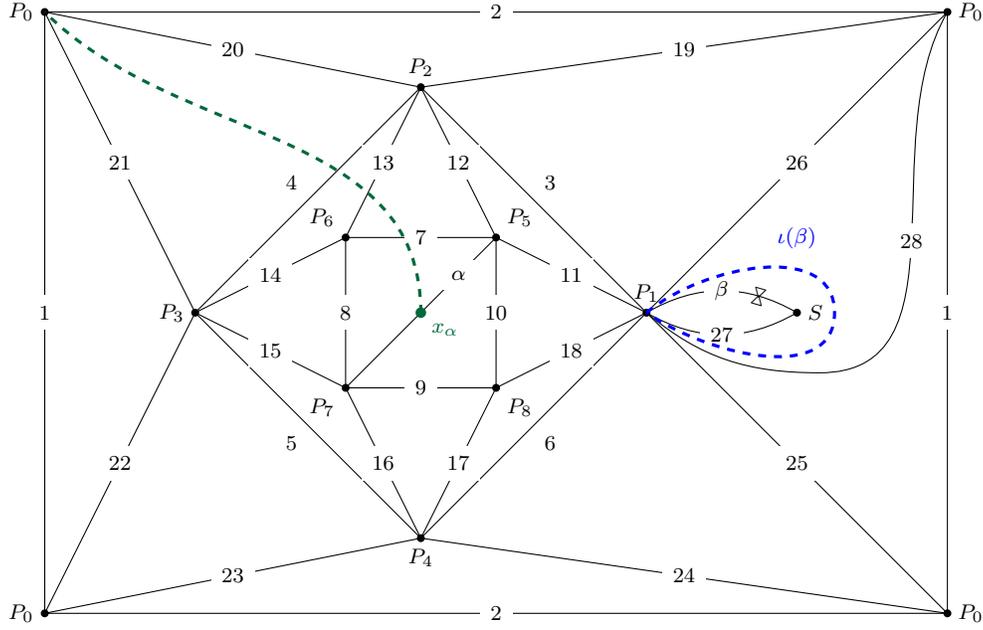

%Note that the starting point of the independence path is not important because regardless of our choice of $x_\alpha$ we can construct an isotopy to an independence path starting at some other point of $\alpha$. That is for every pair of points $x_\alpha$ and $x_\alpha'$ for any independence path for $\alpha$ starting at $x_\alpha$ there exists an isotopic independence path for $\alpha$ starting at $x_\alpha'$. 

\begin{lemma}\label{lem:indpathexists}
For any triangulation $T$, proper subset of punctures $\mathcal{P}$, and any arc $\alpha \in E_T^\mathcal{P}$ there exists an independence path for $\alpha$. 
\end{lemma}
\begin{proof}
Let $T$, $\mathcal{P}$, and $\alpha$ be as above. Note two simple facts about independence paths. First, any arc in $E_T^\mathcal{P}$ that is an edge of a triangle which contains a vertex not in $\mathcal{P}$ has an independence path. Second, if one arc in a triangle has an independence path then every other arc in the triangle that is in $E_T^\mathcal{P}$ has an independence path. 

Now suppose that some arc $\alpha \in E_T^\mathcal{P}$ does not have an independence path. Then any arc that shares a triangle with $\alpha$ also does not have an independence path. Continuing inductively, no arc in $E_T^\mathcal{P}$ is an edge of a triangle which has vertex not in $\mathcal{P}$. But since our marked surface is connected this is only possible if $T = E_T^\mathcal{P}$ or rather every puncture is in $\mathcal{P}$. A contradiction to our stipulation that $\mathcal{P}$ is a proper subset of punctures. 
\end{proof}
\end{definition} 
We define $$\sigma_T^{\mathcal{P}}(\alpha)= \inf \{ \text{number of arcs of } E_T^\mathcal{P} \text{ that are crossed by a separation path for } \alpha \}.$$
Note that $\sigma_T^{\mathcal{P}}$ is well-defined by Lemma \ref{lem:indpathexists} and $\sigma_T^{\mathcal{P}}(\alpha) \geq 0$ where equality holds if and only if $\alpha$ is in a triangle with a vertex not in $\mathcal{P}$.
\begin{construction}[Construction of Sequence for independence]\label{lem:sep}
Index the arcs of $E_T^\mathcal{P}=\{\alpha_i\}_{i=1}^m$ so that $i < j $ if and only if $\sigma_T^{\mathcal{P}}(\alpha_i) \leq \sigma_T^{\mathcal{P}} (\alpha_j)$. Then we define the mutation sequence $$\musep^{\mathcal{P}} := (\alpha_1, \dots, \alpha_n).$$ 
\end{construction}

\begin{example}
Continuing with Example \ref{ex:indpath} the triangulation given in Figure \ref{fig:indpath} we have $\sigma_T^\mathcal{P}(\gamma)=0$ if $\gamma \in \{3,4,5,6,\beta\}$. The function $\sigma_T^\mathcal{P}(\gamma)=1$ for $\gamma \in \{11,12,13,14,15,16,17,18\}$, $\sigma_T^\mathcal{P}(\gamma)=2$ for $\gamma \in \{7,8,9,10\}$ and finally $\sigma_T^\mathcal{P}(\alpha)=3.$ Therefore one possible mutation sequence for $\musep^\mathcal{P}$ would be $$\musep^\mathcal{P}=(3,4,5,6,\beta,11,12,13,14,15,16,17,18,7,8,9,10,\alpha).$$ We can rearrange the order of the arcs in the mutation sequence that take the same value on $\sigma_T^\mathcal{P}$ to obtain other mutation sequences. 
\end{example}

\begin{lemma}\label{lem:ind}
 The set of punctures $\mathcal{P}$ is an independent set in $\musep^\mathcal{P}(T)$.
\end{lemma}
\begin{proof}
We show the claim by inducting on $\max\{\sigma_T^{\mathcal{P}}(\alpha) | \alpha \in E_T^\mathcal{P}\}$. 
Suppose $\sigma_T^{\mathcal{P}}$ takes the value 0 for all arcs in $E_T^\mathcal{P}$. Let $\alpha \in E_T^\mathcal{P}$. Then $\sigma_T^{\mathcal{P}}(\alpha) = 0$ if and only if $\alpha$ is in a triangle with a vertex not in $\mathcal{P}$, so let $R \not \in \mathcal{P}$ be a puncture that is a vertex of a triangle with $\alpha$ as an edge. Since $\alpha$ has both endpoints in $\mathcal{P}$ and is contained in a triangle with $R$ the arc obtained from flipping $\alpha$ will have an endpoint at $R$ so it will not be in $E_{\mu_\alpha(T)}^\mathcal{P}$. So $E_{\musep^\mathcal{P}(T)}^\mathcal{P}$ is empty and therefore $\mathcal{P}$ must be an independent set, so the claim is true when $\max\{\sigma_T^{\mathcal{P}}(\alpha) | \alpha \in E_T^\mathcal{P}\}=0$. 

Assume that the claim is true for $k < \max\{\sigma_T^{\mathcal{P}}(\alpha) | \alpha \in E_T^\mathcal{P}\}$. 
Let $\lambda=(\alpha_1,\dots,\alpha_j)$ be the initial part of $\musep^\mathcal{P}$ that runs over all arcs in $E_T^\mathcal{P}$ where $\sigma_T^{\mathcal{P}}$ takes value 0. Let $\chi$ be the remaining part of $\musep^\mathcal{P}$ so that $\mu_\chi\mu_\lambda=\musep^\mathcal{P}.$ Observe that $\lambda$ cannot be empty since $\mathcal{P}$ is a proper subset of punctures of the surface.
%are in $E_{{\mu_\lambda}(T)}^\mathcal{P}$. 
It follows easily that for any arc $\alpha \in E_{{\mu_\lambda}(T)}^\mathcal{P}$ we have $\sigma_{\mu_\lambda(T)}^\mathcal{P}(\alpha)=\sigma_{T}^\mathcal{P}(\alpha)-1$ since any independence path for such an arc must have crossed one of the $\alpha_\ell$ in $T$, but will no longer cross the mutated arc $\alpha'_\ell$ in $\mu_\lambda(T)$. Furthermore the arcs of $E^\mathcal{P}_{\mu_\lambda(T)}$ are indexed by increasing value of $\sigma_{\mu_\lambda(T)}^\mathcal{P}$ so $\mu_\chi$ is a sequence that could be constructed in Construction \ref{lem:sep} so by our inductive hypothesis $\mathcal{P}$ is in an independent set of punctures in $\mu_\chi(\mu_\lambda(T)).$ But $\mu_\chi\mu_\lambda = \musep^\mathcal{P}$ so we have proven the claim. 

%Now $E_{\mu_\lambda(T)}^\mathcal{P}$ is non-empty by assumption, but now $\sigma^{\mu_{\lambda}(T)}(\alpha) = 0$ for all $\alpha \in E_{\mu_\lambda(T)}^\mathcal{P}$ as they are adjacent to a 
\end{proof}
\begin{lemma}\label{lem:greenind}
If every arc of $E_{T}^\mathcal{P}$ is green, then $\musep$ is a green sequence. 
\end{lemma}
\begin{proof}
If $\mu_k$ is a green mutation then the only vertex that goes from green to red is $k$, and every arc that is mutated, is mutated exactly once in the sequence. Therefore $\musep$ is a green sequence. 
\end{proof}

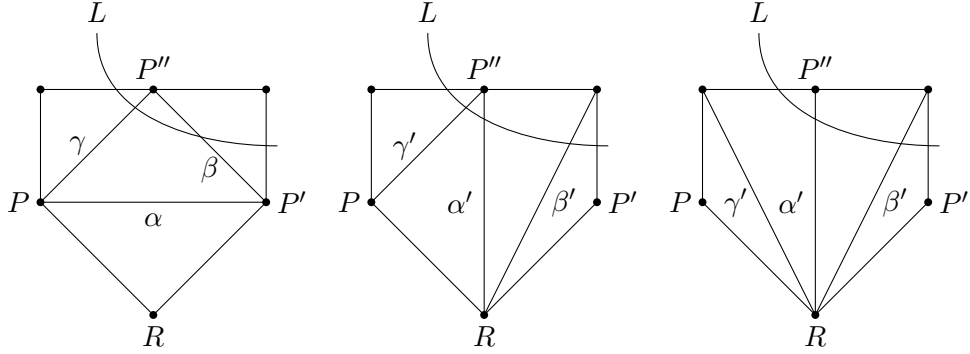
\begin{figure}[t]
\begin{tikzpicture}[x=1.5cm,y=1.5cm, font=\small]
\coordinate (a) at (-1,1);
\coordinate (b) at (1,1);
\coordinate (p2) at (0,1);
\coordinate (p1) at (1,0);
\coordinate (p) at (-1,0);
\coordinate (r) at (0,-1);

\draw (p) to (a) to (p2) to (b) to (p1) to node[below]{$\beta$} (p2) to node[left]{$\gamma$} (p) to (r) to (p1) to node[below]{$\alpha$} (p);
\draw (-.5,1.5) node[above] {$L$} to[out=-90,in=180] (1.1,.5);

\fill (p1) node[right]{$P'$}  circle (1.5pt);
\fill (p2) node[above]{$P''$} circle (1.5pt);
\fill (p) node[left]{$P$} circle (1.5pt);
\fill (r) node[below]{$R$} circle (1.5pt);
\fill (a) circle (1.5pt);
\fill (b) circle (1.5pt);
\end{tikzpicture}
%fig 2
\begin{tikzpicture}[x=1.5cm,y=1.5cm, font=\small]
\coordinate (a) at (-1,1);
\coordinate (b) at (1,1);
\coordinate (p2) at (0,1);
\coordinate (p1) at (1,0);
\coordinate (p) at (-1,0);
\coordinate (r) at (0,-1);

\fill (p1) node[right]{$P'$}  circle (1.5pt);
\fill (p2) node[above]{$P''$} circle (1.5pt);
\fill (p) node[left]{$P$} circle (1.5pt);
\fill (r) node[below]{$R$} circle (1.5pt);
\fill (a) circle (1.5pt);
\fill (b) circle (1.5pt);

\draw (p) to (a) to (p2) to (b) to (p1)to (r)  to (p);
\draw (p) to node[left]{$\gamma'$}(p2);
\draw (r) to  node[left]{$\alpha'$} (p2);
\draw (r) to node[right]{$\beta'$}(b);
\draw (-.5,1.5) node[above] {$L$} to[out=-90,in=180] (1.1,.5);
\end{tikzpicture}
%pic 3
\begin{tikzpicture}[x=1.5cm,y=1.5cm, font=\small]
\coordinate (a) at (-1,1);
\coordinate (b) at (1,1);
\coordinate (p2) at  (0,1);
\coordinate (p1) at (1,0);
\coordinate (p) at (-1,0);
\coordinate (r) at (0,-1);

\draw (p) to (a) to (p2) to (b) to (p1) to (r) to (p);
\draw (r) to node[left]{$\gamma'$}(a);
\draw (r) to  node[left]{$\alpha'$}(p2);
\draw (r) to node[right]{$\beta'$}(b);
\draw (-.5,1.5) node[above] {$L$} to[out=-90,in=180] (1.1,.5);

\fill (p1) node[right]{$P'$}  circle (1.5pt);
\fill (p2) node[above]{$P''$} circle (1.5pt);
\fill (p) node[left]{$P$} circle (1.5pt);
\fill (r) node[below]{$R$} circle (1.5pt);
\fill (a) circle (1.5pt);
\fill (b) circle (1.5pt);
\end{tikzpicture}
\caption{The different cases that arise in the proof of Lemma \ref{lem:cycleok}(1). The left figure is the triangulation $\mu(T)$, the middle and right triangulations are the two possiblilities for $\mu_\beta \dots \mu_\alpha\mu(T).$ The middle is the case when $\beta$ is mutated before $\gamma$ and the right corresponds to when $\gamma$ is mutated before $\beta$.}\label{fig:cycleok}
\end{figure}
\begin{lemma}\label{lem:cycleok}
Assume there are no radial punctures in $\mathcal{P}$. Suppose $\musep^\mathcal{P} = (\gamma_1,\dots,\gamma_n)$, and each $\gamma_i$ is green in a triangulation $T$. Note that $\gamma_i \neq \gamma_j$ for all $i \neq j$. Let $\alpha = \gamma_j$ for some $j = 1, \dots, n$. Let $\mu=(\gamma_1,\dots,\gamma_{j-1}).$ Let $\alpha'$ be the unique new arc obtained from mutating $\alpha$ in $\mu(T)$. By Lemma \ref{lem:ind} $\alpha'$ can have at most one endpoint in $\mathcal{P}$.\begin{enumerate}
\item If $\alpha'$ has exactly one endpoint in $\mathcal{P}$, then $\alpha'$ is green in $\musep^\mathcal{P}(T)$. 
\item If $\alpha'$ is not incident to any puncture in $\mathcal{P}$, then $\alpha'$ is red in $\musep^\mathcal{P}(T)$. 
\end{enumerate}
\end{lemma}
\begin{proof}
Since the arc $\alpha$ was mutated in $\musep^{\mathcal{P}}$ it is in $E^\mathcal{P}_{T}$ and hence has both endpoints in $\mathcal{P}$. Let $P, P' \in \mathcal{P}$ be the endpoints of $\alpha.$ Note that it is possible for $P=P'$. 

Suppose that $\alpha'$ has one endpoint $P'' \in \mathcal{P}$ and its other endpoint $R \not \in\mathcal{P}$. Then there exists a quadrilateral $PP''P'R$ in $\mu(T)$ with diagonals $\alpha$ and $\alpha'.$ Let $\beta$ be the arc counter-clockwise from $\alpha$ in the quadrilateral with both endpoints in $\mathcal{P}$. Then $\beta \in E_T^\mathcal{P}$ and $\beta = \gamma_k$ for some $k > j$. See the left triangulation in Figure \ref{fig:cycleok}. We will show that when $\beta$ is mutated during $\musep^{\mathcal{P}}$ the arc $\alpha'$ will turn green. Suppose that $L$ is some lamination with $b_\beta(T,L)=-1$ and let $\gamma$ is the third arc of the triangle with edges $\alpha$ and $\beta.$ Then by our assumption that $\gamma$ is green $L$ must curve upwards in Figure \ref{fig:cycleok}. It is then easy to see that $b_{\alpha'}(\mu_\beta \dots \mu_\alpha\mu(T),L)=-1$ so $\alpha'$ is green in $ \mu_\beta \dots \mu_\alpha\mu(T)$. Note that we must check both the case when $\gamma$ is mutated before $\beta$ in $\musep^\mathcal{P}$ and the case when $\beta$ is mutated first. We provide a picture of the triangulation $\mu_\beta \dots \mu_\alpha\mu(T)$ in both cases in Figure \ref{fig:cycleok}.  Since $\alpha$ is not mutated again in $\musep^\mathcal{P}$ it will still be green in $\musep^\mathcal{P}(T)$.

%Note that if $\gamma$ is the third arc of the triangle with edges $\alpha$ and $\beta$ it is possible that $\gamma$ is mutated before $\beta$ and there may be a question of whether or not the lamination $L$

Suppose that $\alpha'$ is not incident to any puncture in $\mathcal{P}$. Let $\mu(T)$. Let $R,R' \not\in \mathcal{P}$ be the endpoints of $\alpha'$. Then there exists a quadrilateral $PRP'R'$ in $\mu(T)$ with diagonals $\alpha$ and $\alpha'.$ But notice that each arc composing the boundary of this quadrilateral is not in $E_{T'}^\mathcal{P}$ so none of them will be mutated in the remaining mutations of $\musep.$ Therefore if $L$ is some lamination such that $b_\alpha(T,L)=-1$ then we have $b_{\alpha'}(\musep^{\mathcal{P}}(T),L)=1$ so $\alpha'$ is red in $\mu_\alpha\mu(T)$. Now as none of the arcs in the quadrilateral containing $\alpha'$ are mutated after $\alpha$ in $\musep^\mathcal{P}$ the arc $\alpha$ will remain red in $\musep^\mathcal{P}(T)$.  
\end{proof}

\section{Construction of maximal green sequences for closed surfaces}\label{sec:proof}

Let $\Sigma$ be a closed marked surface of genus zero with at least four punctures, or a closed marked surface of genus at least one with at at least two punctures. It was shown by Ladkani in \cite{ladkani} that any once-punctured closed surface has no maximal green sequence. 

Let $T$ be a triangulation of $\Sigma$. We assume for simplicity that all arcs are tagged plain at all punctures except for radial punctures. We may make this assumption because the corresponding quivers of two triangulations that differ only by the tagging at a puncture are isomorphic. For all $\alpha\in T$ let $\alpha^\circ$ denote the elementary lamination associated to $\alpha$. Fix the multi-lamination $\mathcal{L}=\{\alpha^\circ\}_{\alpha \in T}$. 

\begin{figure}[t]
$$
\begin{xy} 0;<.75pt,0pt>:<0pt,-.65pt>:: 
(0,0) *+{1} ="1",
(25,200) *+{2} ="2",
(400,300) *+{3} ="3",
(400,225) *+{4} ="4",
(350,250) *+{5} ="5",
(300,300) *+{6} ="6",
(300,200) *+{7} ="7",
(400,150) *+{8} ="8",
(400,400) *+{9} ="9",
(275,400) *+{10} ="10",
(275,350) *+{11} ="11",
(175,400) *+{12} ="12",
(0,400) *+{13} ="13",
(125,350) *+{14} ="14",
(50,325) *+{15} ="15",
(225,325) *+{16} ="16",
(250,275) *+{17} ="17",
(125,300) *+{18} ="18",
(215,235) *+{19} ="19",
(150,250) *+{20} ="20",
(250,200) *+{21} ="21",
(300,100) *+{22} ="22",
(400,0) *+{23} ="23",
(225,50) *+{24} ="24",
(225,25) *+{25} ="25",
(200,200) *+{26} ="26",
(135,205) *+{27} ="27",
(75,250) *+{28} ="28",
(95,185) *+{29} ="29",
(225,150) *+{30} ="30",
(50,150) *+{31} ="31",
(75,33) *+{32} ="32",
(50,100) *+{33} ="33",
(150,125) *+{34} ="34",
(100,125) *+{35} ="35",
(91,86) *+{36} ="36",
"1", {\ar"2"},
"13", {\ar"1"},
"1", {\ar"23"},
"24", {\ar"1"},
"25", {\ar"1"},
"2", {\ar"13"},
"2", {\ar"15"},
"28", {\ar"2"},
"3", {\ar"4"},
"5", {\ar"3"},
"6", {\ar"3"},
"3", {\ar"9"},
"4", {\ar"5"},
"4", {\ar"7"},
"8", {\ar"4"},
"5", {\ar"6"},
"7", {\ar"5"},
"6", {\ar"7"},
"9", {\ar"6"},
"7", {\ar"8"},
"22", {\ar"8"},
"8", {\ar"23"},
"9", {\ar"10"},
"11", {\ar"9"},
"10", {\ar"12"},
"12", {\ar"11"},
"13", {\ar"12"},
"12", {\ar"14"},
"14", {\ar"13"},
"15", {\ar"14"},
"14", {\ar"16"},
"16", {\ar"15"},
"15", {\ar"28"},
"16", {\ar"17"},
"18", {\ar"16"},
"17", {\ar"18"},
"19", {\ar"17"},
"17", {\ar"21"},
"18", {\ar"19"},
"20", {\ar"18"},
"19", {\ar"20"},
"21", {\ar"19"},
"20", {\ar"26"},
"27", {\ar"20"},
"21", {\ar"22"},
"26", {\ar"21"},
"23", {\ar"22"},
"22", {\ar"26"},
"23", {\ar"24"},
"23", {\ar"25"},
"26", {\ar"27"},
"27", {\ar"28"},
"29", {\ar"27"},
"28", {\ar"29"},
"30", {\ar"29"},
"29", {\ar"31"},
"31", {\ar"30"},
"30", {\ar"32"},
"34", {\ar"30"},
"33", {\ar"31"},
"31", {\ar"35"},
"32", {\ar"33"},
"32", {\ar"34"},
"36", {\ar"32"},
"35", {\ar"33"},
"33", {\ar"36"},
"35", {\ar"34"},
"34", {\ar"36"},
"36", {\ar"35"},
\end{xy}$$
\caption{The quiver corresponding to the triangulation $T^*$ given in Figure \ref{fig:main_example_triangulation}.}\label{fig:main_example_quiver}

\end{figure}

\begin{figure}[t]
\centering
\begin{tikzpicture}[x=6cm,y=5cm,font=\tiny]
\coordinate (a) at (1.00000000000000,0.000000000000000);
\coordinate (b) at (0.707106781186548,0.707106781186548);
\coordinate (c) at (0.00000000000000,1.00000000000000);
\coordinate (d) at (-1,1);
\coordinate (e) at (-1.6,0);
\coordinate (f) at (-1,-1);
\coordinate (g) at (0,-1.00000000000000);
\coordinate (h) at (0.707106781186548,-0.707106781186548);
\coordinate (i) at (0.0,0);
\coordinate (q5) at (-.1,0.5);
\coordinate (q2) at (-1,0);
\coordinate (q3) at (-.4,-.15);
\coordinate (q4) at (-.5,-.3);
\coordinate (r1) at (-1.1,-.2);
\coordinate (r2) at (-.9,-.2);
\coordinate (r3) at (-1,-.4);
\coordinate (r4) at (-.5,-.9);

\coordinate (mb1) at (.2357,.098-1);
\coordinate (mb2) at (2*.2357,2*.098-1);
\coordinate (mt1) at (.2357,-.098+1);
\coordinate (mt2) at (2*.2357,-2*.098+1);
\coordinate (mt3) at (.855,0.707106781186548/2);
\coordinate (mb3) at (.855,-0.707106781186548/2);

\fill (a) circle (1.5pt);
\fill (b) circle (1.5pt);
\fill (c) circle (1.5pt);
\fill (d) circle (1.5pt);
\fill (e) circle (1.5pt);
\fill (f) circle (1.5pt);
\fill (g) circle (1.5pt);
\fill (h) circle (1.5pt);
\fill (i) circle (2.5pt);
\fill (q5) circle (2.5pt);
\fill (q4) circle (2.5pt);
\fill (q3) circle (2.5pt);
\fill (q2) circle (2.5pt);
\fill (r1) circle (2.5pt);
\fill (r2) circle (2.5pt);
\fill (r3) circle (2.5pt);
\fill (r4) circle (2.5pt);

\draw (c) to node [above] {2} (d) to node [left] {1} (e) to node [left] {2} (f) to node [below] {1} (g);

\draw[dashed] (a) to (b) to (c);
\draw[dashed] (g) to (h) to (a);

\draw (mb1) to node [midway,fill=white]{4} (i) to node [midway,fill=white]{3} (mt1);
\draw (mb2) to node [midway,fill=white]{5} (i) to node [midway,fill=white]{5} (mt2);
\draw (mb3) to node [midway,fill=white]{3} (i) to node [midway,fill=white]{4} (mt3);

\coordinate (mt4) at (.9,.23) ;
\coordinate (mb4) at (.9,-.23);
\draw[bend right] (mt4) to node [left]{5} (mb4);

\draw (i) to[bend right] node [midway,fill=white]{14} (e);
\draw (g) to node [midway,fill=white]{8} (i);
\draw (i) to node [midway,fill=white]{9} (c);
\draw (i) to[out=135, in = 225, min distance=25mm] node [left,fill=white]{12} (c);
\draw (i) to node [midway,fill=white]{10} (q5);
\draw (q5) to node [midway,fill=white]{11} (c);
\draw (e) to node [midway,fill=white]{13} (c);
\draw (f) to[out=125,in=180] node [midway,fill=white]{28} (q2);
\draw (f) to[out=45, in = 0] node [midway,fill=white]{27} (q2);
\draw (i) to node [midway,fill=white]{6} (h);
\draw (i) to node [midway,fill=white]{7} (b);
\draw (q2) to node [midway,fill=white]{15} (e);
\draw (q2) to node [midway,fill=white]{16} (i);
\draw (q3) to node [midway,fill=white]{17} (i);
\draw (q3) to node [midway,fill=white]{19} (q4);
\draw (q3) to[out=155,in=0] node [near start,fill=white]{18} (q2);
\draw (q4) to[out=135,in=0] node [near start,fill=white]{20} (q2);
\draw (q4) to[out=235,in=45] node [midway,fill=white]{26} (f);
\draw (q4) to node [midway,fill=white]{21} (i);
\draw(i) to[out=250,in=45] node [midway,fill=white]{22}(f);
\draw (g) to[out=90,in=45] node [midway,fill=white]{23} (f);
\draw (q2) to node [midway,fill=white]{35} (r1);
\draw (q2) tonode [midway,fill=white]{34} (r2);
\draw (q2) to[out=-15,in = 0,min distance=12mm] node [near end,fill=white]{30} (r3);
\draw (q2) to[out = 195, in = 180,min distance=12mm] node [near end,fill=white]{31} (r3);
\draw (r2) to node [midway,fill=white]{36} (r1);
\draw (r1) to node [midway,fill=white]{33} (r3);
\draw (r2) to node [midway,fill=white,inner sep=0pt,outer sep=1pt]{32} (r3);

\draw (q2) to[out=0,in = 0,min distance=14mm] (-1,-.5);
\draw (q2) to[out = 180, in = 180,min distance=14mm] (-1,-.5);
\draw (-1,-.5) node [fill=white]{29};
\draw (g) to node [midway,fill=white]{25} (r4);
\draw (g) to[bend right] node [near end,sloped,rotate=90]{$\bowtie$} node [midway,fill=white]{24} (r4);

\draw (a) node[right]{{\small $X$}};
\draw (d) node[left]{{\small $X$}};
\draw (b) node[right]{{\small $X$}};
\draw (c) node[right]{{\small $X$}};
\draw (e) node[left]{{\small $X$}};
\draw (f) node[left]{{\small $X$}};
\draw (g) node[right]{{\small $X$}};
\draw (h) node[right]{{\small $X$}};

\draw (.2,0) node {$P_1$};
\draw (q2) node[above] {$P_2$};
\draw (q3) node[above] {$P_3$};
\draw (q4) node[below right] {$P_4$};
\draw (q5) node[above left] {$P_5$};
\draw (r4) node[left]{$S_1$};
\end{tikzpicture}
\caption{A triangulation $T^*$ of the closed genus two surface with 10 punctures. Its corresponding quiver is shown in Figure \ref{fig:main_example_quiver}. We label the arcs of $T$ by their corresponding vertex in the quiver. The punctures in the interior of the monogn labelled 29 are $R_1,R_2,$ and $R_3$ with $R_1$ in the 6 o'clock position and the other following counter-clockwise. }\label{fig:main_example_triangulation}
\end{figure}
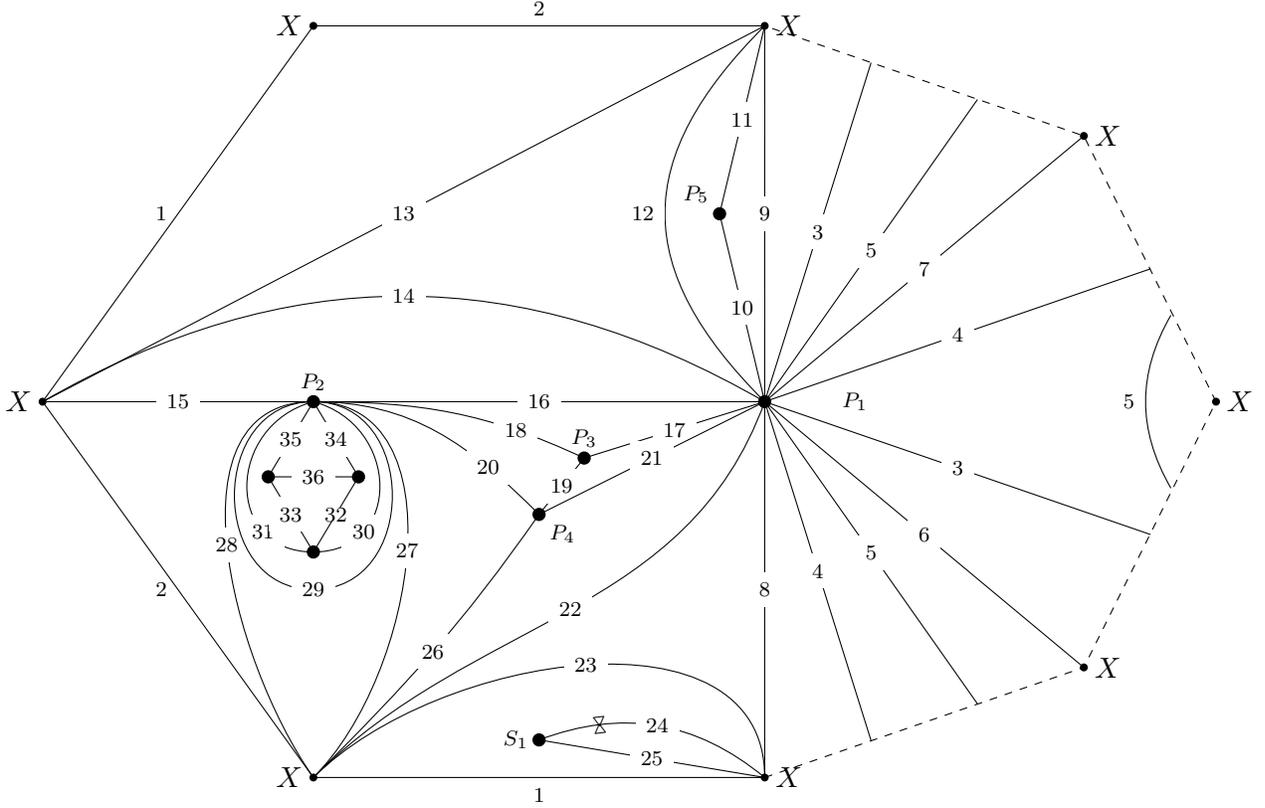

Choose any puncture that is not in the interior of a monogon or a radial puncture and label it $X$. Let $\mathcal{S}$ denote the set of all radial punctures of $\Sigma$ in $T$. Let $\mathcal{M}$ be the set of all punctures of $\Sigma$ that are not in $\mathcal{S}$ and are not $X$. We now define a partition on the set $\mathcal{M}$. 

Let $\mathcal{M}_0$ be the set of all punctures of $\mathcal{M}$ that are not in the interior of any monogon. Define $\mathcal{M}_{i+1}$ to be the punctures of $\mathcal{M}$ that \begin{enumerate}
\item are in the interior of a monogon based at a puncture in $\mathcal{M}_i$;
\item but not in the interior of any monogon based at a puncture in $\mathcal{M} \setminus \bigcup_{j=0}^{i}\mathcal{M}_j.$
\end{enumerate} 

\begin{example}
Consider the triangulation $T^*$ given in Figure \ref{fig:main_example_triangulation}. We choose $X$ to be the puncture in the exterior of the diagram. Then $$\mathcal{S}=\{S_1\}, \mathcal{M}_0=\{P_1, P_2,P_3,P_4,P_5\}, \text{ and } \mathcal{M}_1=\{R_1,R_2,R_3\}.$$
\end{example}

Our maximal green sequence will initially focus on changing the taggings at the punctures of $\mathcal{M}$ and then change the tagging at $X$. The taggings of punctures of $\mathcal{S}$ will be changed during this process. 

\subsection{Independence of punctures.}
Note that it is possible for the set $\mathcal{M}_0$ to be empty. If this is the case then we may skip ahead to Subsection \ref{sec:muI}.

In this step we take $\mathcal{P}= \mathcal{M}_0$ in Construction \ref{lem:sep} and apply $\musep^{\mathcal{M}_0}$ to $T$. No arcs have been mutated yet so $\musep^{\mathcal{M}_0}$ is a green sequence for $T$ by Lemma \ref{lem:greenind}.

\begin{example}\label{ex:main_sep}
In the triangulation $T^*$ in Figure \ref{fig:main_example_triangulation} we have $\mathcal{M}_0= \{P_1, P_2,P_3,P_4,P_5\}$ and it follows that $$E_T^{\mathcal{M}_0} = \{ 3,4,5,10,16,20,21,29,17,18,19\}.$$ Note that $\sigma_T^{\mathcal{M}_0}(\alpha)=0$ for $ \alpha \in \{3,4,5,10,16,20,21,29\},$ and $\sigma_T^{\mathcal{M}_0}(\alpha)=1$ for $\alpha \in \{17,18,19\},$
so
$$\musep^{\mathcal{M}_0}=(3, 4, 5, 10, 16, 17, 18, 19, 20, 21, 29).$$
The order that we mutate the arcs $3,4,5,10,16,20,21,$ and $29$ does not matter, but we will order them by their labels. We adopt this convention in the sequel when we have freedom to do so. See Figure \ref{fig:main_example_triangulation_3} for a picture of the triangulation $\musep^{\mathcal{M}_0}(T^*).$
\end{example}

\subsection{Mutating cycles around punctures.} Label the punctures $\mathcal{M}_0=\{P_i\}_{i=1}^n$ so that in $\musep ^{\mathcal{M}_0}(T)$ we have $P_i$ is not a radial puncture for $i=1, \dots, t$, and $P_i$ is a radial puncture for $i=t+1, \dots, n$.

Now in $\musep^{\mathcal{M}_0}(T)$ we apply $\mucycle^{P_i}$ from Lemma \ref{lem:cyclesurface} to the arcs incident to $P_i$ for $i=1,\dots, t$. By Lemma \ref{lem:ind} we have that there are no loops based at each $P_i$ and all of these mutation sequences $\mucycle^{P_i}$ will be disjoint. 
We define the sequence $$\mucycle^{\mathcal{M}_0}:=\mucycle^{P_t} \dots \mucycle^{P_1}.$$
The mutation sequence $\mucycle^{\mathcal{M}_0}$ is a green sequence for $\musep^{\mathcal{M}_0}(T)$ by Lemma \ref{lem:cycleok}.

%We would also like to draw attention to the fact that after $\mucycle^{\mathcal{M}_0}$ every puncture in $\{P_i\}_{i=1}^t$ has changed its tagging by Lemma \ref{lem:cyclesurface}. 
\begin{example}\label{ex:main_3}
In $\musep^{\mathcal{M}_0}(T^*)$ we have that $P_4$ and $P_5$ are radial punctures so we will not mutate any arcs adjacent to them in $\mucycle^{\mathcal{M}_0}.$ For the other three punctures of ${\mathcal{M}_0}$ we have $$\mucycle^{P_1}= (3, 6, 7, 9, 12, 14, 22, 8, 14, 12, 9, 7, 6, 3),$$ which interchanges arcs 22 and 8;
$$\mucycle^{P_2}= (15, 28, 31, 35, 34, 30, 27, 34, 35, 31, 28, 15),$$ which interchanges arcs 30 and 27; and finally 
$$\mucycle^{P_3}= (16, 20, 21, 16),$$ which interchanges arcs 20 and 21. See Figure \ref{fig:main_example_triangulation_3}.
\end{example}

%Let $L^{\mathcal{M}_0,\circ}=\{\alpha^{\mathcal{M}_0,\circ}\}_{\alpha \in T}$ denote the multi-lamination that coincides with $\mathcal{L}^\circ$ except each lamination spirals in the opposite direction at each puncture of $\{P_i\}_{i=1}^t$. To check whether or not a vertex is green in the next few sections we will use $\mathcal{L}^{\mathcal{M}_0,\circ}$ together with Lemma \ref{lem:altshear}.

\subsection{Mutating back to our initial triangulation.}\label{subsec:indstar}
We now apply a slightly modified version of $(\musep^{\mathcal{M}_0})^{-1}$ to mutate back to our original triangulation. This modification accounts for the punctures $P_i \in {\mathcal{M}_0}$ for $i = t+1, \dots, n$, which we did not apply $\mucycle$ in the previous step and arcs that were interchanged during $\mucycle^{\mathcal{M}_0}.$

For each $P_i \in \mathcal{M}_0$ with $i = t+1, \dots, n$, that is each puncture that is a radial puncture in $\mucycle^{\mathcal{M}_0}\musep^{\mathcal{M}_0}(T)$, let $\alpha_{P_i}$ be the arc mutated during $\musep^{\mathcal{M}_0}$ and $\beta_{P_i}$ the other arc incident to ${P_i}$. We define a sequence $\musep^{\mathcal{M}_0,*}$ that is the same as $\musep^{-1}$ except;
\begin{enumerate}
\item We replace $\alpha_{P_i}$ with $\beta_{P_i}$ in $\musep^{\mathcal{M}_0,*}$ for $i = t+1, \dots, n$;
\item If $\alpha$ appears in $\musep^{\mathcal{M}_0}$ and was interchanged with $\beta$ during $\mucycle^{\mathcal{M}_0}$ then we replace $\alpha$ with $\beta$ in $\musep^{\mathcal{M}_0,*}$. 
\end{enumerate}
Note that for both $\alpha_{P_i}$ and $\beta_{P_i}$ their other endpoint opposite of $P_i$ is $X$, so they were not mutated during $\mucycle^{\mathcal{M}_0}$. 

\begin{example}\label{ex:main_4}
In our running example we have $\musep^{\mathcal{M}_0}=(3, 4, 5, 10, 16, 20, 21, 29, 17, 18, 19).$ To construct $\musep^{\mathcal{M}_0,*}$ we look at $( \musep^{\mathcal{M}_0})^{-1}$ then (1) tells us to replace 10 with 11 and 19 with 26; (2) tells us to replace 20 with 21 and vice versa. That is $$\musep^{\mathcal{M}_0,*}=(26, 18, 17, 29, 20, 21, 16, 11, 5, 4, 3).$$ See Figure \ref{fig:main_example_triangulation_4}.
\end{example}

\begin{lemma}\label{lem:reverse}
$\musep^{\mathcal{M}_0,*}$ is a green sequence for $T'=\mucycle^{\mathcal{M}_0}\musep^{\mathcal{M}_0}(T)$. 
\end{lemma}
\begin{proof}
%If the mutation sequence $\musep^{\mathcal{M}_0}$ is empty there is nothing to check. 
Suppose $\musep^{\mathcal{M}_0,*}=(\alpha_1, \dots, \alpha_n)$. Define $\mu_j=\mu_{\alpha_j}\dots \mu_{\alpha_1}$. Let $\alpha_i^T$ denote the arc in our initial triangulation $T$ that corresponds to the same vertex of the associated quiver as $\alpha_i$. Consider the elementary lamination $(\alpha_{j'}^{T})^\circ$, where $j'=j$ if $\alpha_j$ was not interchanged with another arc during $\mucycle^{\mathcal{M}_0}, $ and $j'=i$ if $\alpha_j$ was interchanged with $\alpha_i$ during $\mucycle^{\mathcal{M}_0}.$ Then we have $$b_{\alpha_j}(\mu_{j-1}(T'),(\alpha_{j'}^{T})^{\circ})
%=b'_\alpha(\mu_n(T'),(\alpha_{j'}^{T})^{\mathcal{M}_0,\circ})
=-1.$$ Therefore at every step of $\musep^{\mathcal{M}_0,*}$ we mutate at a green vertex. 
\end{proof}

It is worthwhile to note that $\musep^{\mathcal{M}_0,*}\musep^{\mathcal{M}_0}\musep^{\mathcal{M}_0}(T)$ coincides with our original triangulation except that the taggings at all of the punctures of $\mathcal{M}_0$ are different. This observation gives rise to the fact that any arcs whose endpoints are exclusively punctures of ${\mathcal{M}_0}$ in $\musep^{\mathcal{M}_0,*}\mucycle^{\mathcal{M}_0}\musep^{\mathcal{M}_0}(T)$ will no longer have to be mutated in our mutation sequence. This also applies for any arcs whose endpoints are a radial puncture in $\mathcal{S}$ and a puncture in $\mathcal{M}_0$. 

\begin{lemma}\label{lem:done1}
Any arc with endpoint exclusively in $\mathcal{M}_0$ in $T'=\musep^{\mathcal{M}_0,*}\mucycle^{\mathcal{M}_0}\musep^{\mathcal{M}_0}(T)$ is done being mutated. 
\end{lemma}
\begin{proof}
Let $\alpha$ be an arc whose endpoints are both in $\mathcal{M}_0$ and let $\alpha^T$ denote the arc in $T$ that corresponds to the same vertex of the associated quiver as $\alpha$. Then $$b_{\gamma}(T',(\alpha^{T})^{\circ})=\begin{cases} 1 & \gamma=\alpha, \\ 0 & \gamma \neq \alpha.\end{cases}$$ Therefore $\alpha$ is done being mutated by Lemma \ref{lem:vertexdone}. 
\end{proof}

\begin{lemma}\label{lem:done2}
If $\alpha$ is an arc with one endpoint in $\mathcal{M}_0$ and the other at a puncture of $\mathcal{S}$ then $\alpha$ is done being mutated. 
\end{lemma}
\begin{proof}
Suppose the end points of $\alpha$ are $P \in \mathcal{M}_0$ and $S \in \mathcal{S}$. Let $\beta$ be the unique other arc with endpoints $P$ and $S$. Then $$b_{\alpha}(T',\beta^{\circ})= 1.$$
Therefore $\alpha$ is done being mutated by Lemma \ref{lem:vertexdone}. 
\end{proof}

We now iteratively apply a similar mutation sequence to each set of punctures $\mathcal{M}_i$ for $i \geq 1.$

\subsection{Punctures in the interior of monogons.}\label{sec:muI}

For $i >0$; \begin{itemize}
\item Let $\musep^{\mathcal{M}_i}$ be a mutation sequence from Construction \ref{lem:sep} with $\mathcal{P}=\mathcal{M}_i$.
\item Take $\mucycle^{\mathcal{M}_i}$ to be the composition of mutation sequences $\mucycle^{P}$ for each cycle around a puncture $P \in \mathcal{M}_i,$ that is not a radial puncture in $ \musep^{\mathcal{M}_i}\musep^{\mathcal{M}_{i-1},*}\mucycle^{\mathcal{M}_{i-1}}\musep^{\mathcal{M}_{i-1}} \cdots \musep^{\mathcal{M}_0,*}\mucycle^{\mathcal{M}_0}\musep^{\mathcal{M}_0}(T).$
\item Construct $\musep^{\mathcal{M}_i,*}$ as we did in Subsection \ref{subsec:indstar}.
\end{itemize} 

Suppose the highest index for nonempty $\mathcal{M}_i$ is $k$. Define $$\mu^\mathcal{M}:=\musep^{\mathcal{M}_k,*}\mucycle^{\mathcal{M}_k}\musep^{\mathcal{M}_k} \cdots \musep^{\mathcal{M}_0,*}\mucycle^{\mathcal{M}_0}\musep^{\mathcal{M}_0}.$$

The proof that each mutation sequence $\musep^{\mathcal{M}_i,*}\mucycle^{\mathcal{M}_i}\musep^{\mathcal{M}_i}$ is a green sequence follows from an identical argument as in the sections above. 

\begin{example}\label{ex:main_5}
In our running example of $T^*$ we have: $$\musep^{\mathcal{M}_1} = (32,33,36).$$
After $\musep^{\mathcal{M}_1}$ we have two degenerate cycles consisting of 27 and 31, and 32 and 35 so
$$ \mucycle^{\mathcal{M}_1} = (27, 31,32, 35).$$
Now in $(\musep^{\mathcal{M}_1})^{-1}$ we replace 36 by 34 as they are incident to a radial puncture, and as 35 was interchanged with 32 we replace 32 with 35 so that $$ \musep^{\mathcal{M}_1,*} = (24,33,35).$$ See Figures \ref{fig:main_example_triangulation_4} and \ref{fig:main_example_triangulation_5}.
\end{example}

\begin{remark}
As mentioned previously, Ladkani showed that once-punctured closed surfaces do not admit quivers with maximal green sequences. In this case the procedure we describe here reduces to mutating all the loops based at the puncture so they are no longer loops, applying Lemma \ref{cycle}, and then mutating the loops back into place. However, our approach fails since every arc in the triangulation is a loop and there is no way to mutate them into arcs that are not loops. 
\end{remark}

\subsection{Changing the tagging at X}
Let $T'=\mu^\mathcal{M}(T).$
Take $\mathcal{P}= \{X\}$ in Construction \ref{lem:sep} to construct a mutation sequence $\musep^X$. Note that this mutation sequence only mutates loops at $X$ and none of these arcs have been mutated yet so they are all green. Therefore $\musep^X$ is a green sequence for $T'$. 
We apply a mutation sequence $\mucycle^X$ to the arcs incident to $X$ in $\musep^X(T')$. This changes the tagging at puncture $X.$ The fact that $\mucycle^X$ is a green sequence for $\musep^X(T')$ follows from Lemma \ref{lem:cycleok}.

We also define a mutation sequence $\musep^{X,*}$ in a similar way as we did above. However since $X$ is the only puncture in our independent set we do not create any new radial punctures during $\musep^X$. Therefore when constructing $\musep^{X,*}$ we only need to apply the modification from above which replaces arcs that were interchanged during $\mucycle^X$. The fact that $\musep^{P,*}$ is green sequence for $\mucycle^P\musep^P(T')$ follows the same proof as in Lemma \ref{lem:reverse}.

\begin{example}\label{ex:main_6}
We construct $\musep^X$, $\mucycle^X$ and $\musep^{X,*}$ for $\mu^\mathcal{M}(T^*)$ similar to as we did above. We first make $X$ independent of itself by applying the mutation sequence $$\musep^X = (2,13, 23, 24, 1).$$ Note that $24$ is mutated here because $\iota(24)$ is a loop based at $X$. Then we apply the Cycle Lemma to the arcs incident to $X$ to obtain the sequence $$\mucycle^X = (6, 7, 22, 25, 23, 15, 14, 2, 8, 19, 30, 28, 12, 10, 9, 12, 28, 30, 19, 8, 2, 14, 15, 23, 25, 22, 7, 6).$$ There is no modification to make to $(\musep^{X})^{-1}$ to obtain $\musep^{X,*}$ because arcs 10 and 11 are not mutated during $\musep^{X}$ so $$\musep^{X,*} = (1, 24, 23, 13, 2).$$ See Figures \ref{fig:main_example_triangulation_5} and \ref{fig:main_example_triangulation_6}.
\end{example}

\subsection{Maximal Green Sequence}
Putting together the previous lemmas we have the following theorem. 
\begin{theorem}\label{thm:mgsclosed}
For any triangulation $T$ of a closed marked surface with at least two punctures the sequence $$\musep^{X,*}\mucycle^X\musep^X\mu^\mathcal{M}$$ is a maximal green sequence for $Q_T$.
\end{theorem}
\begin{proof}
From our work above we have shown that this mutation sequence is a green sequence. It remains to be shown that it is in fact maximal. This follows from the same kind of arguments made at the end of Subsection \ref{subsec:indstar} in the proofs of Lemmas \ref{lem:done1} and \ref{lem:done2}. 

We have that $\musep^{X,*}\mucycle^X\musep^X\mu^\mathcal{M}(T)$ coincides with our original triangulation $T$ except that the tagging of arcs differs at every puncture of $\Sigma$. Therefore if the underlying untagged arc of $\alpha \in \musep^{X,*}\mucycle^X\musep^X\mu^\mathcal{M}(T)$ coincides with the underlying untagged arc of $\beta \in T$ then $$b_\alpha(\musep^{X,*}\mucycle^X\musep^X\mu^\mathcal{M}(T),\beta^\circ) = 1.$$ Therefore every arc in $\musep^{X,*}\mucycle^X\musep^X\mu^\mathcal{M}(T)$ is red. 
\end{proof}

\section{Existence of maximal green sequences for surfaces with nonempty boundary}\label{sec:boundary}
\begin{figure}
\centering
\begin{minipage}{.5\textwidth}
\begin{tikzpicture}[x=.8cm,y=.5cm,font=\tiny]
\coordinate (a) at (0,0);
\coordinate (b) at (6,0);
\coordinate (c) at (6,6);
\coordinate (d) at (0,6);
\coordinate (g) at (2,2);
\coordinate (f) at (4,2);
\coordinate (e) at (3,4);
\fill (a) circle (1.5pt);
\fill (b) circle (1.5pt);
\fill (c) circle (1.5pt);
\fill (d) circle (1.5pt);
\fill (e) circle (1.5pt);
\fill (f) circle (1.5pt);
\fill (g) circle (1.5pt);

\draw (a) to node [midway, below] {\tiny{2}} (b);
\draw (b) to node [midway, right] {\tiny{1}} (c);
\draw (c) to node [midway, above] {\tiny{2}} (d);
\draw (d) to node [midway, left] {\tiny{1}} (a);

\draw (d) to node [midway,fill=white]{3} (e);
\draw (e) to node [midway,fill=white]{4} (c);
\draw (e) to[bend left] node [midway,fill=white]{5} (b);
\draw (f) to node [midway,fill=white]{6} (b);
\draw (f) to node [midway,fill=white]{7} (a);
\draw (g) to node [midway,fill=white]{8} (a);
\draw (g) to node [midway,fill=white]{9} (d);
\draw (g) to (e);
\draw (e) to (f);
\draw (g) to (f);

\draw (a) node[left]{{\small $P$}};
\draw (d) node[left]{{\small $P$}};
\draw (b) node[right]{{\small $P$}};
\draw (c) node[right]{{\small $P$}};

\draw[pattern=north west lines] (g) -- (e) -- (f) -- (g);
\end{tikzpicture}
\end{minipage}
\begin{minipage}{.5\textwidth}
\begin{xy} 0;<.3pt,0pt>:<0pt,-.25pt>:: 
(0,0) *+{1} ="0",
(250,225) *+{2} ="1",
(100,175) *+{3} ="2",
(100,75) *+{4} ="3",
(400,75) *+{5} ="4",
(400,225) *+{6} ="5",
(400,400) *+{7} ="6",
(100,400) *+{8} ="7",
(100,300) *+{9} ="8",
%(200,250) *+{10} ="9",
%(300,225) *+{11} ="10",
%(300,300) *+{12} ="11",
"3", {\ar"0"},
"0", {\ar"4"},
"7", {\ar"0"},
"0", {\ar"8"},
"2", {\ar"1"},
"1", {\ar"3"},
"5", {\ar"1"},
"1", {\ar"6"},
"3", {\ar"2"},
"8", {\ar"2"},
%"2", {\ar"9"},
"4", {\ar"3"},
"4", {\ar"5"},
%"10", {\ar"4"},
"6", {\ar"5"},
%"5", {\ar"10"},
"6", {\ar"7"},
%"11", {\ar"6"},
"8", {\ar"7"},
%"7", {\ar"11"},
%"9", {\ar"8"},
%"9", {\ar"10"},
%"11", {\ar"9"},
%"10", {\ar"11"},
\end{xy}
\end{minipage}
\caption{A triangulation $T^*$ of $\Sigma^*=\{1,1,1,\{3\}\}$ (left), and the corresponding quiver (right) that we extend to the triangulation in Figure \ref{ex:1}. Note that this quiver is an induced subquiver of the one in Figure \ref{ex:1}.}\label{ex:1b}
\end{figure}
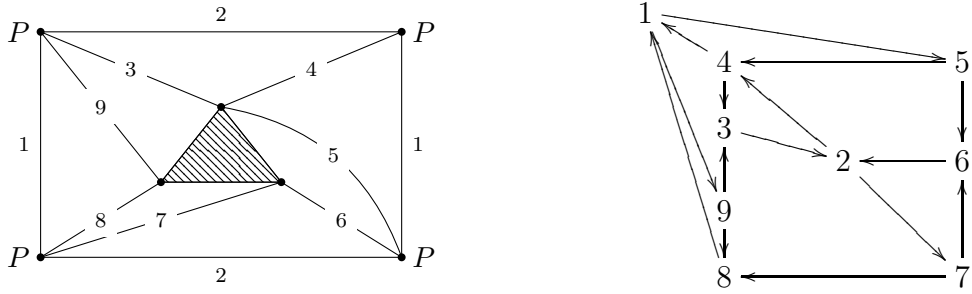
 We first recall a theorem from Muller and the definition of an induced subquiver.

 \begin{definition}
Given a subset $V$ of vertices of a quiver $Q,$ the \textbf{induced subquiver}, is the quiver with vertex set $V$ and edges consisting of the edges between pairs of vertices in $V$ that are in $Q$. 
\end{definition}
\begin{theorem}{\cite[Lemma 1.4.1]{muller}}\label{thm:subquiver}
If a quiver admits a maximal green sequence, then any induced subquiver admits a maximal green sequence. 
\end{theorem}
By the previous theorem it suffices to show that the quiver obtained from any triangulation of a surface with boundary is an induced subquiver of a quiver from a closed surface.

Suppose $\Sigma$ is a marked surface with boundary. We construct a marked surface without boundary $\overline{\Sigma}$ by gluing disks to the boundary components of $\Sigma$. To be more precise, for each boundary component $b_i$ of $\Sigma$ with $m_i$ marked points we glue a disk $D_i$ with $m_i$ marked points on its boundary and a single puncture if $m_i=1$ or $2$.

If $T$ is a triangulation of $\Sigma$ we can extend it to a triangulation $\overline{T}$ of $\overline{\Sigma}$. Let $B$ be a set of arcs of $\overline{\Sigma}$ that are isotopic to sections of the boundary component $\Sigma$. Then consider the collection of arcs $T \cup B$. This is a partial triangulation of the surface $\overline \Sigma$. We add more arcs to this collection to obtain a triangulation $\overline T$ of $\overline{\Sigma}$. Furthermore, by deleting vertices corresponding to the arcs of $\overline T \setminus T$ from $Q_{\overline{T}}$ we see that resulting quiver is $Q_T.$ That is $Q_T$ is an induced subquiver of $Q_{\overline{T}}$. To summarize we have the following lemma. 
\begin{lemma} \label{lem:subquiver}
Let $T$ be any triangulation of $\Sigma$. Let $\overline{\Sigma}$ be as above. Then the quiver $Q_T$ is an induced subquiver of the quiver $Q_{\overline{T}}$ corresponding to the triangulation $\overline{T}$ of $\overline{\Sigma}$. 
\end{lemma}

\begin{theorem}\label{thm:bdy}
The quiver $Q_T$ has a maximal green sequence. 
\end{theorem}
\begin{proof}
By Theorem \ref{thm:mgsclosed} $Q_{\overline{T}}$ has a maximal green sequence and by Lemma \ref{lem:subquiver} $Q_T$ is an induced subquiver of $Q_{\overline{T}}$. Therefore by Theorem \ref{thm:subquiver} the quiver $Q_T$ has a maximal green sequence. 
\end{proof}
\section{Exceptional quivers of finite mutation type}\label{sec:exceptional}
Recall, a quiver $Q$ is said to be of finite mutation type if $\mut(Q)$ is finite. Felikson, Shapiro, and Tumarkin showed that every quiver of finite mutation type has rank 2, arises from a triangulation of marked surface, or is in one of 11 exceptional mutation classes \cite{felikson}. As we mentioned in the introduction every rank 2 quiver has a simple maximal green sequence. The 11 exceptional mutation classes are represented by the quivers $\mathbb{E}_6,\mathbb{E}_7,\mathbb{E}_8,\widetilde{\mathbb{E}_6},\widetilde{\mathbb{E}_7},\widetilde{\mathbb{E}_8},\mathbb{E}_6^{(1,1)},\mathbb{E}_7^{(1,1)},\mathbb{E}_8^{(1,1)}, \mathbb{X}_6,$ and $\mathbb{X}_7.$ 

For four of these cases it is known whether or not these quivers have a maximal green sequence. 

\begin{theorem}{\cite[Theorem 4.1]{brustle}}\label{thm:finite}
Every quiver mutation equivalent to $\mathbb{E}_6,\mathbb{E}_7,$ and $\mathbb{E}_8$ has a maximal green sequence. 
\end{theorem}

\begin{theorem}{\cite{seven}}\label{thm:x7}
Neither quiver in the mutation class of $\mathbb{X}_7$ has a maximal green sequence. 
\end{theorem}

We used the cluster algebra package in Sage developed by Gregg Musiker and Christian Stump to compute an explicit maximal green sequence for every quiver in the remaining 7 exceptional mutation classes. The maximal green sequences can be found on the authors webpage \cite{webpage}.
\begin{theorem}\label{thm:main2}
If $Q$ is a quiver that is mutation equivalent to $\widetilde{\mathbb{E}_6},\widetilde{\mathbb{E}_7},\widetilde{\mathbb{E}_8},\mathbb{E}_6^{(1,1)},\mathbb{E}_7^{(1,1)},\mathbb{E}_8^{(1,1)}$ or $\mathbb{X}_6,$ then $Q$ has a maximal green sequence. 
\end{theorem}

Combining Theorems \ref{thm:mgsclosed}, \ref{thm:bdy}, \ref{thm:finite}, \ref{thm:x7}, and \ref{thm:main2} and the fact about rank 2 quivers we have a complete classification of all quivers of finite mutation type which have a maximal green sequence. 
\begin{theorem}\label{thm:main3}
If $Q$ is a quiver of finite mutation type, then $Q$ has a maximal green sequence unless it arises from a triangulation of a once-punctured closed surface, or is one of the two quivers in the mutation class of $\mathbb{X}_7$.
\end{theorem}

\section{Acknowledgements}
The author would like to thank Kyungyong Lee for many helpful discussions, Pierre-Guy Plamondon for translating the abstract of the original work \cite{mills} to French, and Khrystyna Serhiyenko and Greg Muller for helpful suggestions to improve the manuscript.

\begin{figure}[t]
\centering
\begin{tikzpicture}[x=6cm,y=4.8cm,font=\tiny]
\coordinate (a) at (1.00000000000000,0.000000000000000);
\coordinate (b) at (0.707106781186548,0.707106781186548);
\coordinate (c) at (0.00000000000000,1.00000000000000);
\coordinate (d) at (-1,1);
\coordinate (e) at (-1.6,0);
\coordinate (f) at (-1,-1);
\coordinate (g) at (0,-1.00000000000000);
\coordinate (h) at (0.707106781186548,-0.707106781186548);
\coordinate (i) at (0.0,0);
\coordinate (q5) at (0,0.4);
\coordinate (q2) at (-1,0);
\coordinate (q3) at (-.4,-.15);
\coordinate (q4) at (-.5,-.3);
\coordinate (r1) at (-1.1,-.2);
\coordinate (r2) at (-.9,-.2);
\coordinate (r3) at (-1,-.4);
\coordinate (r4) at (-.5,-.9);

\coordinate (mb1) at (.2357,.098-1);
\coordinate (mb2) at (2*.2357,2*.098-1);
\coordinate (mt1) at (.2357,-.098+1);
\coordinate (mt2) at (2*.2357,-2*.098+1);
\coordinate (mt3) at (.855,0.707106781186548/2);
\coordinate (mb3) at (.855,-0.707106781186548/2);

\fill (a) circle (1.5pt);
\fill (b) circle (1.5pt);
\fill (c) circle (1.5pt);
\fill (d) circle (1.5pt);
\fill (e) circle (1.5pt);
\fill (f) circle (1.5pt);
\fill (g) circle (1.5pt);
\fill (h) circle (1.5pt);
\fill (i) circle (2.5pt);
\fill (q5) circle (2.5pt);
\fill (q4) circle (2.5pt);
\fill (q3) circle (2.5pt);
\fill (q2) circle (2.5pt);
\fill (r1) circle (2.5pt);
\fill (r2) circle (2.5pt);
\fill (r3) circle (2.5pt);
\fill (r4) circle (2.5pt);

\draw[color=darksagegreen] (c) to node [above] {2} (d) to node [left] {1} (e) to node [left] {2} (f) to node [below] {1} (g);

\draw[dashed] (a) to (b);
\draw[color=darksagegreen][color=red] (b) to node [above] {5} (c);
\draw[dashed] (g) to (h);
\draw[color=darksagegreen][color=red] (h) to node [right] {5} (a);

%\draw[color=darksagegreen] (mb1) to node [midway,fill=white]{4} (i) to node [midway,fill=white]{3} (mt1);
%\draw[color=darksagegreen] (mb2) to node [midway,fill=white]{5} (i) to node [midway,fill=white]{5} (mt2);
%\draw[color=darksagegreen] (mb3) to node [midway,fill=white]{3} (i) to node [midway,fill=white]{4} (mt3);
\draw[color=darksagegreen] (i) to node [midway,fill=white]{3} (mb2); 
\draw[color=darksagegreen] (mt3) to node [midway,fill=white]{3} (h);
\draw[color=darksagegreen][color=red] (b) to node [midway,fill=white]{4} (h);
\coordinate (mt4) at (.9,.23) ;
\coordinate (mb4) at (.9,-.23);
%\draw[color=darksagegreen][bend right] (mt4) to node [left]{5} (mb4);

\draw[color=darksagegreen] (i) to[bend right] node [midway,fill=white]{14} (e);
\draw[color=darksagegreen] (g) to node [midway,fill=white]{8} (i);
\draw[color=darksagegreen] (i) to[bend right] node [midway,fill=white]{9} (c);
\draw[color=darksagegreen] (i) to[out=135, in = 225, min distance=25mm] node [left,fill=white]{12} (c);
\draw[color=darksagegreen][color=red] (c) to[bend right] node [midway,fill=white]{10} node[near end,sloped, rotate=90]{$\bowtie$} (q5);
\draw[color=darksagegreen] (q5) to node [midway,fill=white]{11} (c);
\draw[color=darksagegreen] (e) to node [midway,fill=white]{13} (c);
\draw[color=darksagegreen] (f) to[out=125,in=180] node [midway,fill=white]{28} (q2);
\draw[color=darksagegreen] (f) to[out=45, in = 0] node [midway,fill=white]{27} (q2);
\draw[color=darksagegreen] (i) to node [midway,fill=white]{6} (h);
\draw[color=darksagegreen] (i) to node [midway,fill=white]{7} (b);
\draw[color=darksagegreen] (q2) to node [midway,fill=white]{15} (e);
\draw[color=darksagegreen] (q3) to[out=90, in = 5] node [midway,fill=white]{16} (e);
\draw[color=darksagegreen][color=red] (e) to[out=15,in=145] (-.25,0) node [fill=white]{17} to[out=0,in=25] (f) ;
\draw[color=darksagegreen][color=red] (q4) to[out=-90, in= 45] node [midway,fill=white,outer sep=0pt] {19} node[near start,sloped, rotate=90]{$\bowtie$}(f);
\draw[color=darksagegreen][color=red] (e) to[out=5,in=165] (-.6,0) node [fill=white]{18} to[out=235,in=45] (f);
\draw[color=darksagegreen] (q3) to[out=200, in = 45] node [near start,fill=white]{20} (f);
\draw[color=darksagegreen] (q4) to[out=235,in=45] node [midway,fill=white]{26} (f);
\draw[color=darksagegreen] (q3) to[out=300, in = 35] node [midway,fill=white]{21} (f);
\draw[color=darksagegreen](i) to[out=250,in=25] node [midway,fill=white]{22}(f);
\draw[color=darksagegreen] (g) to[out=90,in=25] node [midway,fill=white]{23} (f);
\draw[color=darksagegreen] (q2) to node [midway,fill=white]{35} (r1);
\draw[color=darksagegreen] (q2) tonode [midway,fill=white]{34} (r2);
\draw[color=darksagegreen] (q2) to[out=-15,in = 0,min distance=12mm] node [near end,fill=white]{30} (r3);
\draw[color=darksagegreen] (q2) to[out = 195, in = 180,min distance=12mm] node [near end,fill=white]{31} (r3);
\draw[color=darksagegreen] (r2) to node [midway,fill=white]{36} (r1);
\draw[color=darksagegreen] (r1) to node [midway,fill=white]{33} (r3);
\draw[color=darksagegreen] (r2) to node [midway,fill=white]{32} (r3);

%\draw[color=darksagegreen] (q2) to[out=0,in = 0,min distance=14mm] (-1,-.5);
\draw[color=darksagegreen][color=red] (r3) to node [midway,fill=white]{29}(f) ;
%\draw[color=darksagegreen] (-1,-.5);
\draw[color=darksagegreen] (g) to node [midway,fill=white]{25} (r4);
\draw[color=darksagegreen] (g) to[bend right] node [near end,sloped,rotate=90]{$\bowtie$} node [midway,fill=white]{24} (r4);

\draw (a) node[right]{{\small $X$}};
\draw (d) node[left]{{\small $X$}};
\draw (b) node[right]{{\small $X$}};
\draw (c) node[right]{{\small $X$}};
\draw (e) node[left]{{\small $X$}};
\draw (f) node[left]{{\small $X$}};
\draw (g) node[right]{{\small $X$}};
\draw (h) node[right]{{\small $X$}};

\draw (.2,0) node {$P_1$};
\draw (q2) node[above] {$P_2$};
\draw (q3) node[above right] {$P_3$};
\draw (q4) node[below right] {$P_4$};
\draw (q5) node[above left] {$P_5$};
\draw (r4) node[left]{$S_1$};
\end{tikzpicture}
%\caption{The triangulation $\musep^{\mathcal{M}_0}(T^*)$ obtained from applying the mutation sequence $\musep^{\mathcal{M}_0}$ from Example \ref{ex:main_sep} to the triangulation $T^*$ in Figure \ref{fig:main_example_triangulation}.}\label{fig:main_example_triangulation_2}
%\end{figure}

%\begin{figure}
%\centering
\begin{tikzpicture}[x=6cm,y=4.8cm,font=\tiny]
\coordinate (a) at (1.00000000000000,0.000000000000000);
\coordinate (b) at (0.707106781186548,0.707106781186548);
\coordinate (c) at (0.00000000000000,1.00000000000000);
\coordinate (d) at (-1,1);
\coordinate (e) at (-1.6,0);
\coordinate (f) at (-1,-1);
\coordinate (g) at (0,-1.00000000000000);
\coordinate (h) at (0.707106781186548,-0.707106781186548);
\coordinate (i) at (0.0,0);
\coordinate (q5) at (0,0.4);
\coordinate (q2) at (-1,0);
\coordinate (q3) at (-.4,-.15);
\coordinate (q4) at (-.5,-.3);
\coordinate (r1) at (-1.1,-.2);
\coordinate (r2) at (-.9,-.2);
\coordinate (r3) at (-1,-.4);
\coordinate (r4) at (-.5,-.9);

\coordinate (mb1) at (.2357,.098-1);
\coordinate (mb2) at (2*.2357,2*.098-1);
\coordinate (mt1) at (.2357,-.098+1);
\coordinate (mt2) at (2*.2357,-2*.098+1);
\coordinate (mt3) at (.855,0.707106781186548/2);
\coordinate (mb3) at (.855,-0.707106781186548/2);

\fill (a) circle (1.5pt);
\fill (b) circle (1.5pt);
\fill (c) circle (1.5pt);
\fill (d) circle (1.5pt);
\fill (e) circle (1.5pt);
\fill (f) circle (1.5pt);
\fill (g) circle (1.5pt);
\fill (h) circle (1.5pt);
\fill (i) circle (2.5pt);
\fill (q5) circle (2.5pt);
\fill (q4) circle (2.5pt);
\fill (q3) circle (2.5pt);
\fill (q2) circle (2.5pt);
\fill (r1) circle (2.5pt);
\fill (r2) circle (2.5pt);
\fill (r3) circle (2.5pt);
\fill (r4) circle (2.5pt);

\draw[color=darksagegreen] (c) to node [above] {2} (d) to node [left] {1} (e) to node [left] {2} (f) to node [below] {1} (g);

\draw[dashed] (a) to (b);
\draw[color=darksagegreen] (b) to node [above] {5} (c);
\draw[dashed] (g) to (h);
\draw[color=darksagegreen] (h) to node [right] {5} (a);

%\draw[color=darksagegreen] (mb1) to node [midway,fill=white]{4} (i) to node [midway,fill=white]{3} (mt1);
%\draw[color=darksagegreen] (mb2) to node [midway,fill=white]{5} (i) to node [midway,fill=white]{5} (mt2);
%\draw[color=darksagegreen] (mb3) to node [midway,fill=white]{3} (i) to node [midway,fill=white]{4} (mt3);
\draw[color=darksagegreen][color=red] (i) to node [midway,fill=white]{3} node[near start,sloped, rotate=90]{$\bowtie$}(mb2); 
\draw[color=darksagegreen] (mt3) to node [midway,fill=white]{3} (h);
\draw[color=darksagegreen] (b) to node [midway,fill=white]{4} (h);
\coordinate (mt4) at (.9,.23) ;
\coordinate (mb4) at (.9,-.23);
%\draw[color=darksagegreen][bend right] (mt4) to node [left]{5} (mb4);

\draw[color=darksagegreen][color=red] (i) to[bend right] node [midway,fill=white]{14} node[near start,sloped, rotate=90]{$\bowtie$}(e);
\draw[color=darksagegreen][color=red] (g) to node [midway,fill=white]{22} node[near end,sloped, rotate=90]{$\bowtie$}(i);
\draw[color=darksagegreen][color=red] (i) to[bend right] node [midway,fill=white]{9} node[near start,sloped, rotate=90]{$\bowtie$}(c);
\draw[color=darksagegreen][color=red] (i) to[out=135, in = 225, min distance=25mm] node [left,fill=white]{12} node[near start,sloped, rotate=90]{$\bowtie$}(c);
\draw[color=darksagegreen] (c) to[bend right] node [midway,fill=white]{10} node[near end,sloped, rotate=90]{$\bowtie$} (q5);
\draw[color=darksagegreen] (q5) to node [midway,fill=white]{11} (c);
\draw[color=darksagegreen] (e) to node [midway,fill=white]{13} (c);
\draw[color=darksagegreen][color=red] (f) to[out=125,in=180] node [midway,fill=white]{28} node[near end,sloped, rotate=90]{$\bowtie$} (q2);
\draw[color=darksagegreen][color=red] (f) to[out=45, in = 0] node [midway,fill=white]{30} node[near end,sloped, rotate=90]{$\bowtie$} (q2);
\draw[color=darksagegreen][color=red] (i) to node [midway,fill=white]{6} node[near start,sloped, rotate=90]{$\bowtie$} (h);
\draw[color=darksagegreen][color=red] (i) to node [midway,fill=white]{7} node[near start,sloped, rotate=90]{$\bowtie$}(b);
\draw[color=darksagegreen][color=red] (q2) to node [midway,fill=white]{15} node[near start,sloped, rotate=90]{$\bowtie$} (e);
\draw[color=darksagegreen][color=red] (q3) to[out=90, in = 5] node [midway,fill=white]{16} node[near start,sloped, rotate=90]{$\bowtie$} (e);
\draw[color=darksagegreen] (e) to[out=15,in=145] (-.25,0) node [fill=white]{17} to[out=0,in=25] (f) ;
\draw[color=darksagegreen] (q4) to[out=-90, in= 45] node [midway,fill=white,outer sep=0pt]{19} node[near start,sloped, rotate=90]{$\bowtie$} (f);
\draw[color=darksagegreen] (e) to[out=5,in=165] (-.6,0) node [fill=white]{18} to[out=235,in=45] (f);
\draw[color=darksagegreen][color=red] (q3) to[out=200, in = 45] node [near start,fill=white]{21} node[pos=.1,sloped, rotate=90]{$\bowtie$} (f);
\draw[color=darksagegreen] (q4) to[out=235,in=45] node [midway,fill=white]{26} (f);
\draw[color=darksagegreen][color=red] (q3) to[out=300, in = 35] node [midway,fill=white]{20} node[near start,sloped, rotate=90]{$\bowtie$} (f);
\draw[color=darksagegreen][color=red] (i) to[out=250,in=25] node [midway,fill=white]{8} node[near start,sloped, rotate=90]{$\bowtie$}(f);
\draw[color=darksagegreen] (g) to[out=90,in=25] node [midway,fill=white]{23} (f);
\draw[color=darksagegreen][color=red] (q2) to node [midway,fill=white]{35} node[near start,sloped, rotate=90]{$\bowtie$} (r1);
\draw[color=darksagegreen][color=red] (q2) tonode [midway,fill=white]{34}  node[near start,sloped, rotate=90]{$\bowtie$} (r2);
\draw[color=darksagegreen][color=red] (q2) to[out=-15,in = 0,min distance=12mm] node [near end,fill=white]{27} node[near start,sloped, rotate=90]{$\bowtie$}(r3);
\draw[color=darksagegreen][color=red] (q2) to[out = 195, in = 180,min distance=12mm] node [near end,fill=white]{31} node[near start,sloped, rotate=90]{$\bowtie$} (r3);
\draw[color=darksagegreen] (r2) to node [midway,fill=white]{36} (r1);
\draw[color=darksagegreen] (r1) to node [midway,fill=white]{33} (r3);
\draw[color=darksagegreen] (r2) to node [midway,fill=white]{32} (r3);

%\draw[color=darksagegreen] (q2) to[out=0,in = 0,min distance=14mm] (-1,-.5);
\draw[color=darksagegreen] (r3) to node [midway,fill=white]{29} (f) ;
%\draw[color=darksagegreen] (-1,-.5);
\draw[color=darksagegreen] (g) to node [midway,fill=white]{25} (r4);
\draw[color=darksagegreen] (g) to[bend right] node [near end,sloped,rotate=90]{$\bowtie$} node [midway,fill=white]{24} (r4);

\draw (a) node[right]{{\small $X$}};
\draw (d) node[left]{{\small $X$}};
\draw (b) node[right]{{\small $X$}};
\draw(c) node[right]{{\small $X$}};
\draw(e) node[left]{{\small $X$}};
\draw (f) node[left]{{\small $X$}};
\draw (g) node[right]{{\small $X$}};
\draw (h) node[right]{{\small $X$}};

\draw (.2,0) node {$P_1$};
\draw (q2) node[above] {$P_2$};
\draw (q3) node[above] {$P_3$};
\draw (q4) node[below right] {$P_4$};
\draw (q5) node[above left] {$P_5$};
\draw (r4) node[left]{$S_1$};
\end{tikzpicture}
\caption{The top triangulation is $\musep^{\mathcal{M}_0}(T^*)$ obtained from applying the mutation sequence $\musep^{\mathcal{M}_0}$ from Example \ref{ex:main_sep} to the triangulation $T^*$ in Figure \ref{fig:main_example_triangulation}.The bottom triangulation $\mucycle^{\mathcal{M}_0}\musep^{\mathcal{M}_0}(T^*)$ obtained from applying the mutation sequence $\musep^{\mathcal{M}_0}$ from Example \ref{ex:main_3}.}\label{fig:main_example_triangulation_3}
\end{figure}

\begin{figure}[t]
\centering
\begin{tikzpicture}[x=6cm,y=4.8cm,font=\tiny]
\coordinate (a) at (1.00000000000000,0.000000000000000);
\coordinate (b) at (0.707106781186548,0.707106781186548);
\coordinate (c) at (0.00000000000000,1.00000000000000);
\coordinate (d) at (-1,1);
\coordinate (e) at (-1.6,0);
\coordinate (f) at (-1,-1);
\coordinate (g) at (0,-1.00000000000000);
\coordinate (h) at (0.707106781186548,-0.707106781186548);
\coordinate (i) at (0.0,0);
\coordinate (q5) at (-.1,0.5);
\coordinate (q2) at (-1,0);
\coordinate (q3) at (-.4,-.15);
\coordinate (q4) at (-.5,-.3);
\coordinate (r1) at (-1.1,-.2);
\coordinate (r2) at (-.9,-.2);
\coordinate (r3) at (-1,-.4);
\coordinate (r4) at (-.5,-.9);

\coordinate (mb1) at (.2357,.098-1);
\coordinate (mb2) at (2*.2357,2*.098-1);
\coordinate (mt1) at (.2357,-.098+1);
\coordinate (mt2) at (2*.2357,-2*.098+1);
\coordinate (mt3) at (.855,0.707106781186548/2);
\coordinate (mb3) at (.855,-0.707106781186548/2);

\fill (a) circle (1.5pt);
\fill (b) circle (1.5pt);
\fill (c) circle (1.5pt);
\fill (d) circle (1.5pt);
\fill (e) circle (1.5pt);
\fill (f) circle (1.5pt);
\fill (g) circle (1.5pt);
\fill (h) circle (1.5pt);
\fill (i) circle (2.5pt);
\fill (q5) circle (2.5pt);
\fill (q4) circle (2.5pt);
\fill (q3) circle (2.5pt);
\fill (q2) circle (2.5pt);
\fill (r1) circle (2.5pt);
\fill (r2) circle (2.5pt);
\fill (r3) circle (2.5pt);
\fill (r4) circle (2.5pt);

\draw[color=darksagegreen] (c) to node [above] {2} (d) to node [left] {1} (e) to node [left] {2} (f) to node [below] {1} (g);

\draw[dashed] (a) to (b) to (c);
\draw[dashed] (g) to (h) to (a);

\draw[color=darksagegreen][color=red] (mb1) to node [midway,fill=white]{4} node[near end,sloped, rotate=90]{$\bowtie$} (i) to node [midway,fill=white]{3} node[near start,sloped, rotate=90]{$\bowtie$} (mt1);
\draw[color=darksagegreen][color=red] (mb2) to node [midway,fill=white] {5} node[near end,sloped, rotate=90]{$\bowtie$} (i) to node [midway,fill=white]{5} node[near start,sloped, rotate=90]{$\bowtie$} (mt2);
\draw[color=darksagegreen][color=red] (mb3) to node [midway,fill=white]{3} node[near end,sloped, rotate=90]{$\bowtie$} (i) to node [midway,fill=white]{4} node[near start,sloped, rotate=90]{$\bowtie$} (mt3);

\coordinate (mt4) at (.9,.23) ;
\coordinate (mb4) at (.9,-.23);
\draw[color=darksagegreen][bend right,color=red] (mt4) to node [left]{5} (mb4);

\draw[color=darksagegreen][color=red] (i) to[bend right] node [midway,fill=white]{14} node[near start,sloped, rotate=90]{$\bowtie$} (e);
\draw[color=darksagegreen][color=red] (g) to node [midway,fill=white]{22} node[near end,sloped, rotate=90]{$\bowtie$} (i);
\draw[color=darksagegreen] (i) to node [midway,fill=white]{9} node[near start,sloped, rotate=90]{$\bowtie$} (c);
\draw[color=darksagegreen] (i) to[out=135, in = 225, min distance=25mm] node [left,fill=white]{12} node[near start,sloped, rotate=90]{$\bowtie$} (c);
\draw[color=darksagegreen][color=red] (i) to node[near start,sloped, rotate=90]{$\bowtie$} node[near end,sloped, rotate=90]{$\bowtie$} (q5);
\draw[color=darksagegreen] (q5) to node [midway,fill=white]{10} node[near start,sloped, rotate=90]{$\bowtie$} (c);
\draw[color=darksagegreen] (e) to node [midway,fill=white]{13} (c);
\draw[color=darksagegreen][color=red] (f) to[out=125,in=180] node [midway,fill=white]{28} node[near end,sloped, rotate=90]{$\bowtie$} (q2);
\draw[color=darksagegreen] (f) to[out=45, in = 0] node [midway,fill=white]{30} node[near end,sloped, rotate=90]{$\bowtie$} (q2);
\draw[color=darksagegreen] (i) to node [midway,fill=white]{6} node[near start,sloped, rotate=90]{$\bowtie$} (h);
\draw[color=darksagegreen] (i) to node [midway,fill=white]{7} node[near start,sloped, rotate=90]{$\bowtie$} (b);
\draw[color=darksagegreen] (q2) to node [midway,fill=white]{15} node[near start,sloped, rotate=90]{$\bowtie$} (e);
\draw[color=darksagegreen][color=red] (q2) to node[near start,sloped, rotate=90]{$\bowtie$} node[near end,sloped, rotate=90]{$\bowtie$} (i);
\draw[color=darksagegreen] [color=red](q3) to node[near start,sloped, rotate=90]{$\bowtie$} node[near end,sloped, rotate=90]{$\bowtie$} (i);
\draw[color=darksagegreen][color=red] (q3) to node[near start,sloped, rotate=90]{$\bowtie$} node[near end,sloped, rotate=90]{$\bowtie$} (q4);
\draw[color=darksagegreen][color=red] (q3) to[out=155,in=0] node[pos=.1,sloped, rotate=90]{$\bowtie$} node[near end,sloped, rotate=90]{$\bowtie$} (q2);
\draw[color=darksagegreen][color=red] (q4) to[out=135,in=0] node[pos=.1,sloped, rotate=90]{$\bowtie$}  node[near end,sloped, rotate=90]{$\bowtie$} (q2);
\draw[color=darksagegreen] (q4) to[out=235,in=45] node [midway,fill=white]{19} node[near start,sloped, rotate=90]{$\bowtie$} (f);
\draw[color=darksagegreen][color=red] (q4) to node[near start,sloped, rotate=90]{$\bowtie$} node[near end,sloped, rotate=90]{$\bowtie$} (i);
\draw[color=darksagegreen](i) to[out=250,in=45] node [midway,fill=white]{8} node[near start,sloped, rotate=90]{$\bowtie$} (f);
\draw[color=darksagegreen] (g) to[out=90,in=45] node [midway,fill=white]{23} (f);
\draw[color=darksagegreen][color=red] (q2) to node [midway,fill=white]{35} node[near start,sloped, rotate=90]{$\bowtie$} (r1);
\draw[color=darksagegreen][color=red] (q2) to node [midway,fill=white]{34} node[near start,sloped, rotate=90]{$\bowtie$} (r2);
\draw[color=darksagegreen][color=red] (q2) to[out=-15,in = 0,min distance=12mm] node [near end,fill=white]{27} (r3);
\draw[color=darksagegreen] (q2) to[out = 195, in = 180,min distance=12mm] node [near end,fill=white]{31} (r3);
\draw[color=darksagegreen] (r2) to node [midway,fill=white]{36} (r1);
\draw[color=darksagegreen] (r1) to node [midway,fill=white]{33} (r3);
\draw[color=darksagegreen] (r2) to node [midway,fill=white,inner sep=0pt,outer sep=1pt]{32} (r3);

\draw[color=darksagegreen][color=red] (q2) to[out=0,in = 0,min distance=14mm] node[near start,sloped, rotate=90]{$\bowtie$} (-1,-.5);
\draw[color=darksagegreen][color=red] (q2) to[out = 180, in = 180,min distance=14mm] node[near start,sloped, rotate=90]{$\bowtie$} (-1,-.5);
\draw[color=darksagegreen][color=red] (-1,-.5) node [fill=white]{29};
\draw[color=darksagegreen] (g) to node [midway,fill=white]{25} (r4);
\draw[color=darksagegreen] (g) to[bend right] node [near end,sloped,rotate=90]{$\bowtie$} node [midway,fill=white]{24} (r4);

\draw (a) node[right]{{\small $X$}};
\draw (d) node[left]{{\small $X$}};
\draw (b) node[right]{{\small $X$}};
\draw (c) node[right]{{\small $X$}};
\draw (e) node[left]{{\small $X$}};
\draw (f) node[left]{{\small $X$}};
\draw (g) node[right]{{\small $X$}};
\draw (h) node[right]{{\small $X$}};

\draw (.2,0) node {$P_1$};
\draw (q2) node[above] {$P_2$};
\draw (q3) node[above] {$P_3$};
\draw (q4) node[below right] {$P_4$};
\draw (q5) node[above left] {$P_5$};
\draw (r4) node[left]{$S_1$};
\end{tikzpicture}
\begin{tikzpicture}[x=1.5cm,y=1.5cm,font=\tiny]
\coordinate (p2) at (0,2);
\coordinate (l) at (0,-2);
\coordinate (r1) at (0,-1);
\coordinate (r2) at (.5,0.5);
\coordinate (r3) at (-.5,0.5);

\draw[color=darksagegreen] (r1) to node [fill=white]{32} (r2);
\draw[color=darksagegreen] (r2) to node [fill=white]{36} (r3);
\draw[color=darksagegreen] (r3) to node [fill=white]{33} (r1);
\draw[color=darksagegreen] (r2) to node [fill=white]{34} node[near end,sloped, rotate=90]{$\bowtie$} (p2);
\draw[color=darksagegreen] (r3) to node [fill=white]{35} node[near end,sloped, rotate=90]{$\bowtie$} (p2);
\draw[color=darksagegreen] (r1) to[out=0,in=-10,min distance=14mm] node [fill=white]{27} node[near end,sloped, rotate=90]{$\bowtie$} (p2);
\draw[color=darksagegreen] (r1) to[out=180,in=190,min distance=14mm] node [fill=white]{31} node[near end,sloped, rotate=90]{$\bowtie$} (p2);

\draw[color=red] (p2) to[out=0,in=0,min distance=24mm] node[near start,sloped, rotate=90]{$\bowtie$} (l) to[out=180,in=180,min distance=24mm] node[near end,sloped, rotate=90]{$\bowtie$} (p2);

\draw[color=darksagegreen][color=red] (l) node [fill=white]{29};

\fill (p2) node[above]{$P_2$} circle (1.5pt);
\fill (r1) circle (1.5pt);
\fill (r2) circle (1.5pt);
\fill (r3) circle (1.5pt);
\end{tikzpicture}
%% second
\begin{tikzpicture}[x=1.5cm,y=1.5cm,font=\tiny]
\coordinate (p2) at (0,2);
\coordinate (l) at (0,-2);
\coordinate (r1) at (0,-1.5);
\coordinate (r2) at (0,.5);
\coordinate (r3) at (0,-.5);
\coordinate (l2) at (0,-1);

\draw[color=darksagegreen] (p2) to[out=-10,in=0] node[near end,fill=white]{32} node[near start,sloped, rotate=90]{$\bowtie$} node[near start,sloped, rotate=90]{$\bowtie$} (r3);
\draw[color=darksagegreen][color=red] (r2) to[bend right] node [fill=white]{36} node[near start,sloped, rotate=90]{$\bowtie$} node[near end,sloped, rotate=90]{$\bowtie$} (p2);
\draw[color=darksagegreen][color=red] (p2) to[out=180,in=180] node[near start,sloped, rotate=90]{$\bowtie$} (l2) node [fill=white]{33} to[out=0,in=0] node[near end,sloped, rotate=90]{$\bowtie$} (p2);
\draw[color=darksagegreen][color=red] (r2) to[bend left] node [fill=white]{34} node[near end,sloped, rotate=90]{$\bowtie$} (p2);
\draw[color=darksagegreen] (r3) to[out=180,in=190] node [near start, fill=white]{35} node[near end,sloped, rotate=90]{$\bowtie$} (p2);
\draw[color=darksagegreen] (r1) to[out=0,in=-10,min distance=14mm] node [near start,fill=white]{27} (p2);
\draw[color=darksagegreen] (r1) to[out=180,in=190,min distance=14mm] node [near start, fill=white]{31} (p2);

\draw[color=red] (p2) to[out=0,in=0,min distance=24mm] node[near start,sloped, rotate=90]{$\bowtie$} (l) to[out=180,in=180,min distance=24mm] node[near end,sloped, rotate=90]{$\bowtie$} (p2);
\draw[color=darksagegreen] (l) node [fill=white]{29};

\fill (p2) node[above]{$P_2$} circle (1.5pt);
\fill (r1) circle (1.5pt);
\fill (r2) circle (1.5pt);
\fill (r3) circle (1.5pt);
\end{tikzpicture}
%%3rd
\begin{tikzpicture}[x=1.5cm,y=1.5cm,font=\tiny]
\coordinate (p2) at (0,2);
\coordinate (l) at (0,-2);
\coordinate (r1) at (0,-1.5);
\coordinate (r2) at (0,.5);
\coordinate (r3) at (0,-.5);
\coordinate (l2) at (0,-1);

\draw[color=darksagegreen][color=red] (p2) to[out=-10,in=0] node[near end,fill=white]{35} node[pos=.9,sloped, rotate=90]{$\bowtie$} node[near start,sloped, rotate=90]{$\bowtie$} node[near start,sloped, rotate=90]{$\bowtie$} (r3);
\draw[color=darksagegreen][color=red] (r2) to[bend right] node [fill=white]{36} node[near start,sloped, rotate=90]{$\bowtie$} node[near end,sloped, rotate=90]{$\bowtie$} (p2);
\draw[color=darksagegreen] (p2) to[out=180,in=180] node[near start,sloped, rotate=90]{$\bowtie$} (l2) node [fill=white]{33} to[out=0,in=0] node[near end,sloped, rotate=90]{$\bowtie$} (p2);
\draw[color=darksagegreen] (r2) to[bend left] node [fill=white]{34} node[near end,sloped, rotate=90]{$\bowtie$} (p2);
\draw[color=darksagegreen][color=red] (r3) to[out=180,in=190] node [near start, fill=white]{32} node[near end,sloped, rotate=90]{$\bowtie$} node[pos=.1,sloped, rotate=90]{$\bowtie$} (p2);
\draw[color=darksagegreen][color=red] (r1) to[out=0,in=-10,min distance=14mm] node [near start,fill=white]{31} node[near end,sloped, rotate=90]{$\bowtie$} node[pos=.1,sloped, rotate=90]{$\bowtie$} (p2);
\draw[color=darksagegreen][color=red] (r1) to[out=180,in=190,min distance=14mm] node [near start, fill=white]{27} node[near end,sloped, rotate=90]{$\bowtie$} node[pos=.1,sloped, rotate=90]{$\bowtie$} (p2);

\draw[color=red] (p2) to[out=0,in=0,min distance=24mm] node[near start,sloped, rotate=90]{$\bowtie$} (l) to[out=180,in=180,min distance=24mm] node[near end,sloped, rotate=90]{$\bowtie$} (p2);

\draw[color=red] (l) node [fill=white]{29};

\fill (p2) node[above]{$P_2$} circle (1.5pt);
\fill (r1) circle (1.5pt);
\fill (r2) circle (1.5pt);
\fill (r3) circle (1.5pt);
\end{tikzpicture}

\caption{The top triangulation is $\musep^{\mathcal{M}_0,*}\mucycle^{\mathcal{M}_0}\musep^{\mathcal{M}_0}(T^*)$ where $\musep^{\mathcal{M}_0,*}$ is defined in Example \ref{ex:main_4}. The bottom is a close up view of the interior of the monogon 29 (left) followed by the $\musep^{\mathcal{M}_1}\musep^{\mathcal{M}_0,*}\mucycle^{\mathcal{M}_0}\musep^{\mathcal{M}_0}(T^*)$ (center), and $\mucycle^{\mathcal{M}_1}e\musep^{\mathcal{M}_1}\musep^{\mathcal{M}_0,*}\mucycle^{\mathcal{M}_0}\musep^{\mathcal{M}_0}(T^*)$ (right). These mutation sequences are defined in Example \ref{ex:main_5}.}\label{fig:main_example_triangulation_4}
\end{figure}
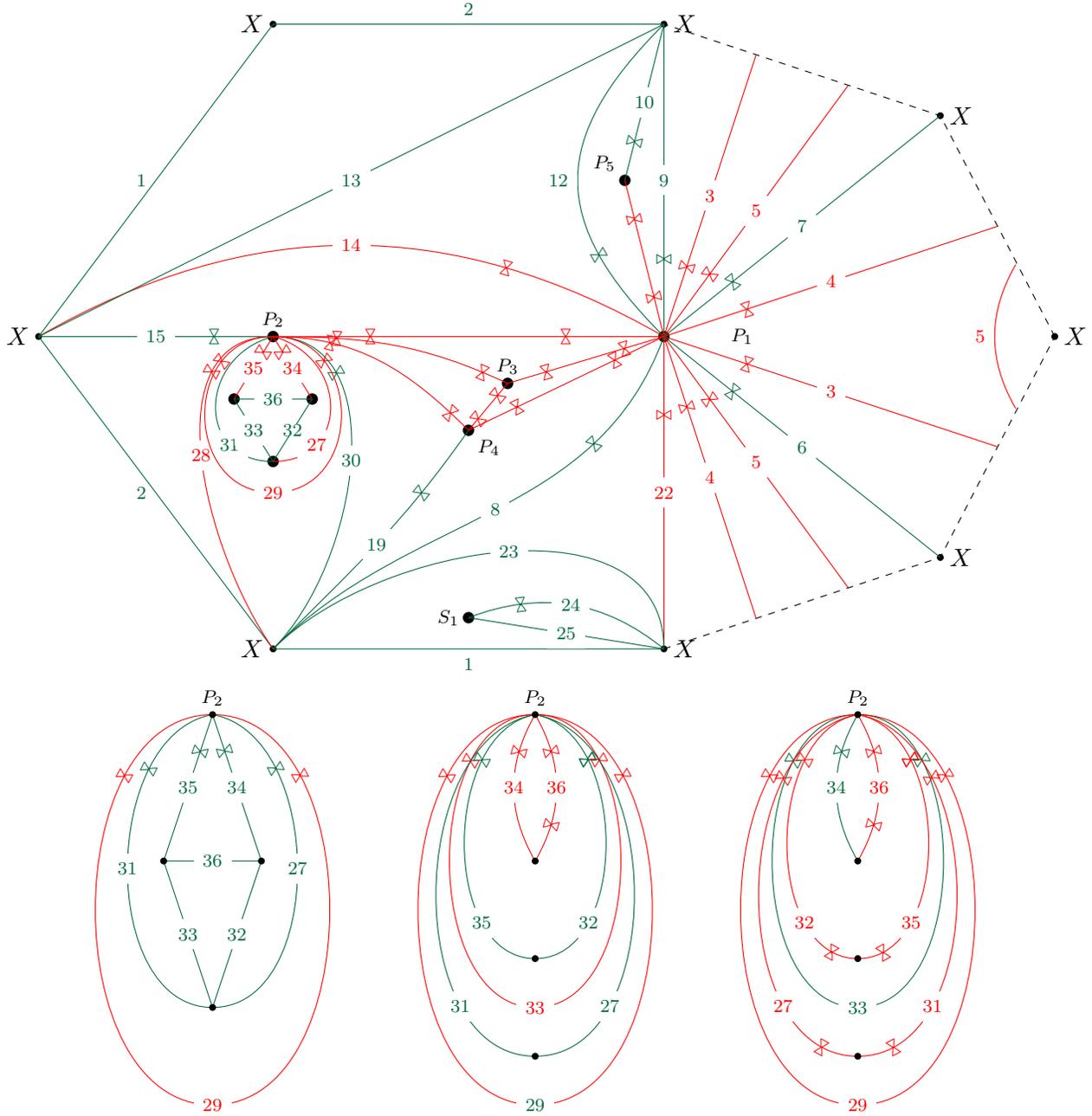

\begin{figure}[t]
\centering
\begin{tikzpicture}[x=6cm,y=4.8cm,font=\tiny]
\coordinate (a) at (1.00000000000000,0.000000000000000);
\coordinate (b) at (0.707106781186548,0.707106781186548);
\coordinate (c) at (0.00000000000000,1.00000000000000);
\coordinate (d) at (-1,1);
\coordinate (e) at (-1.6,0);
\coordinate (f) at (-1,-1);
\coordinate (g) at (0,-1.00000000000000);
\coordinate (h) at (0.707106781186548,-0.707106781186548);
\coordinate (i) at (0.0,0);
\coordinate (q5) at (-.1,0.5);
\coordinate (q2) at (-1,0);
\coordinate (q3) at (-.4,-.15);
\coordinate (q4) at (-.5,-.3);
\coordinate (r1) at (-1.1,-.2);
\coordinate (r2) at (-.9,-.2);
\coordinate (r3) at (-1,-.4);
\coordinate (r4) at (-.5,-.9);

\coordinate (mb1) at (.2357,.098-1);
\coordinate (mb2) at (2*.2357,2*.098-1);
\coordinate (mt1) at (.2357,-.098+1);
\coordinate (mt2) at (2*.2357,-2*.098+1);
\coordinate (mt3) at (.855,0.707106781186548/2);
\coordinate (mb3) at (.855,-0.707106781186548/2);

\fill (a) circle (1.5pt);
\fill (b) circle (1.5pt);
\fill (c) circle (1.5pt);
\fill (d) circle (1.5pt);
\fill (e) circle (1.5pt);
\fill (f) circle (1.5pt);
\fill (g) circle (1.5pt);
\fill (h) circle (1.5pt);
\fill (i) circle (2.5pt);
\fill (q5) circle (2.5pt);
\fill (q4) circle (2.5pt);
\fill (q3) circle (2.5pt);
\fill (q2) circle (2.5pt);
\fill (r1) circle (2.5pt);
\fill (r2) circle (2.5pt);
\fill (r3) circle (2.5pt);
\fill (r4) circle (2.5pt);

\draw[color=darksagegreen] (c) to node [above] {2} (d) to node [left] {1} (e) to node [left] {2} (f) to node [below] {1} (g);

\draw[dashed] (a) to (b) to (c);
\draw[dashed] (g) to (h) to (a);

\draw[color=darksagegreen][color=red] (mb1) to node [midway,fill=white]{4} node[near end,sloped, rotate=90]{$\bowtie$} (i) to node [midway,fill=white]{3} node[near start,sloped, rotate=90]{$\bowtie$} (mt1);
\draw[color=darksagegreen][color=red] (mb2) to node [midway,fill=white] {5} node[near end,sloped, rotate=90]{$\bowtie$} (i) to node [midway,fill=white]{5} node[near start,sloped, rotate=90]{$\bowtie$} (mt2);
\draw[color=darksagegreen][color=red] (mb3) to node [midway,fill=white]{3} node[near end,sloped, rotate=90]{$\bowtie$} (i) to node [midway,fill=white]{4} node[near start,sloped, rotate=90]{$\bowtie$} (mt3);

\coordinate (mt4) at (.9,.23) ;
\coordinate (mb4) at (.9,-.23);
\draw[color=darksagegreen][bend right,color=red] (mt4) to node [left]{5} (mb4);

\draw[color=darksagegreen][color=red] (i) to[bend right] node [midway,fill=white]{14} node[near start,sloped, rotate=90]{$\bowtie$} (e);
\draw[color=darksagegreen][color=red] (g) to node [midway,fill=white]{22} node[near end,sloped, rotate=90]{$\bowtie$} (i);
\draw[color=darksagegreen] (i) to node [midway,fill=white]{9} node[near start,sloped, rotate=90]{$\bowtie$} (c);
\draw[color=darksagegreen] (i) to[out=135, in = 225, min distance=25mm] node [left,fill=white]{12} node[near start,sloped, rotate=90]{$\bowtie$} (c);
\draw[color=darksagegreen][color=red] (i) to node[near start,sloped, rotate=90]{$\bowtie$} node[near end,sloped, rotate=90]{$\bowtie$} (q5);
\draw[color=darksagegreen] (q5) to node [midway,fill=white]{10} node[near start,sloped, rotate=90]{$\bowtie$} (c);
\draw[color=darksagegreen] (e) to node [midway,fill=white]{13} (c);
\draw[color=darksagegreen][color=red] (f) to[out=125,in=180] node [midway,fill=white]{28} node[near end,sloped, rotate=90]{$\bowtie$} (q2);
\draw[color=darksagegreen] (f) to[out=45, in = 0] node [midway,fill=white]{30} node[near end,sloped, rotate=90]{$\bowtie$} (q2);
\draw[color=darksagegreen] (i) to node [midway,fill=white]{6} node[near start,sloped, rotate=90]{$\bowtie$} (h);
\draw[color=darksagegreen] (i) to node [midway,fill=white]{7} node[near start,sloped, rotate=90]{$\bowtie$} (b);
\draw[color=darksagegreen] (q2) to node [midway,fill=white]{15} node[near start,sloped, rotate=90]{$\bowtie$} (e);
\draw[color=darksagegreen][color=red] (q2) to node[near start,sloped, rotate=90]{$\bowtie$} node[near end,sloped, rotate=90]{$\bowtie$} (i);
\draw[color=darksagegreen] [color=red](q3) to node[near start,sloped, rotate=90]{$\bowtie$} node[near end,sloped, rotate=90]{$\bowtie$} (i);
\draw[color=darksagegreen][color=red] (q3) to node[near start,sloped, rotate=90]{$\bowtie$} node[near end,sloped, rotate=90]{$\bowtie$} (q4);
\draw[color=darksagegreen][color=red] (q3) to[out=155,in=0] node[pos=.1,sloped, rotate=90]{$\bowtie$} node[near end,sloped, rotate=90]{$\bowtie$} (q2);
\draw[color=darksagegreen][color=red] (q4) to[out=135,in=0] node[pos=.1,sloped, rotate=90]{$\bowtie$}  node[near end,sloped, rotate=90]{$\bowtie$} (q2);
\draw[color=darksagegreen] (q4) to[out=235,in=45] node [midway,fill=white]{19} node[near start,sloped, rotate=90]{$\bowtie$} (f);
\draw[color=darksagegreen][color=red] (q4) to node[near start,sloped, rotate=90]{$\bowtie$} node[near end,sloped, rotate=90]{$\bowtie$} (i);
\draw[color=darksagegreen](i) to[out=250,in=45] node [midway,fill=white]{8} node[near start,sloped, rotate=90]{$\bowtie$} (f);
\draw[color=darksagegreen] (g) to[out=90,in=45] node [midway,fill=white]{23} (f);

\draw[color=darksagegreen][color=red] (q2) to node[near start,sloped, rotate=90]{$\bowtie$} node[near end,sloped, rotate=90]{$\bowtie$} (r1);
\draw[color=red] (q2) to node[near start,sloped, rotate=90]{$\bowtie$} node[near end,sloped, rotate=90]{$\bowtie$} (r2);
\draw[color=red] (q2) to[out=-15,in = 0,min distance=12mm] node[near start,sloped, rotate=90]{$\bowtie$} node[near end,sloped, rotate=90]{$\bowtie$} (r3);
\draw[color=darksagegreen][color=red] (q2) to[out = 195, in = 180,min distance=12mm] node[near start,sloped, rotate=90]{$\bowtie$} node[near end,sloped, rotate=90]{$\bowtie$}(r3);
\draw[color=darksagegreen][color=red] (r2) to node[near start,sloped, rotate=90]{$\bowtie$} node[near end,sloped, rotate=90]{$\bowtie$}(r1);
\draw[color=darksagegreen][color=red] (r1) to node[near start,sloped, rotate=90]{$\bowtie$} node[near end,sloped, rotate=90]{$\bowtie$}(r3);
\draw[color=darksagegreen][color=red] (r2) to node[near start,sloped, rotate=90]{$\bowtie$} node[near end,sloped, rotate=90]{$\bowtie$} (r3);

\draw[color=darksagegreen][color=red] (q2) to[out=0,in = 0,min distance=14mm] node[near start,sloped, rotate=90]{$\bowtie$} (-1,-.5);
\draw[color=darksagegreen][color=red] (q2) to[out = 180, in = 180,min distance=14mm] node[near start,sloped, rotate=90]{$\bowtie$} (-1,-.5);
\draw[color=darksagegreen][color=red] (-1,-.5) node [fill=white]{29};
\draw[color=darksagegreen] (g) to node [midway,fill=white]{25} (r4);
\draw[color=darksagegreen] (g) to[bend right] node [near end,sloped,rotate=90]{$\bowtie$} node [midway,fill=white]{24} (r4);

\draw (a) node[right]{{\small $X$}};
\draw (d) node[left]{{\small $X$}};
\draw (b) node[right]{{\small $X$}};
\draw (c) node[right]{{\small $X$}};
\draw (e) node[left]{{\small $X$}};
\draw (f) node[left]{{\small $X$}};
\draw (g) node[right]{{\small $X$}};
\draw (h) node[right]{{\small $X$}};

\draw (.2,0) node {$P_1$};
\draw (q2) node[above] {$P_2$};
\draw (q3) node[above] {$P_3$};
\draw (q4) node[below right] {$P_4$};
\draw (q5) node[above left] {$P_5$};
\draw (r4) node[left]{$S_1$};
\end{tikzpicture}

\begin{tikzpicture}[x=6cm,y=4.8cm,font=\tiny]
\coordinate (a) at (1.00000000000000,0.000000000000000);
\coordinate (b) at (0.707106781186548,0.707106781186548);
\coordinate (c) at (0.00000000000000,1.00000000000000);
\coordinate (d) at (-1,1);
\coordinate (e) at (-1.6,0);
\coordinate (f) at (-1,-1);
\coordinate (g) at (0,-1.00000000000000);
\coordinate (h) at (0.707106781186548,-0.707106781186548);
\coordinate (i) at (0.0,0);
\coordinate (q5) at (-.1,0.5);
\coordinate (q2) at (-1,0);
\coordinate (q3) at (-.4,-.15);
\coordinate (q4) at (-.5,-.3);
\coordinate (r1) at (-1.1,-.2);
\coordinate (r2) at (-.9,-.2);
\coordinate (r3) at (-1,-.4);
\coordinate (r4) at (-.13,-.5);

\coordinate (mb1) at (.2357,.098-1);
\coordinate (mb2) at (2*.2357,2*.098-1);
\coordinate (mt1) at (.2357,-.098+1);
\coordinate (mt2) at (2*.2357,-2*.098+1);
\coordinate (mt3) at (.855,0.707106781186548/2);
\coordinate (mb3) at (.855,-0.707106781186548/2);

\fill (a) circle (1.5pt);
\fill (b) circle (1.5pt);
\fill (c) circle (1.5pt);
\fill (d) circle (1.5pt);
\fill (e) circle (1.5pt);
\fill (f) circle (1.5pt);
\fill (g) circle (1.5pt);
\fill (h) circle (1.5pt);
\fill (i) circle (2.5pt);
\fill (q5) circle (2.5pt);
\fill (q4) circle (2.5pt);
\fill (q3) circle (2.5pt);
\fill (q2) circle (2.5pt);
\fill (r1) circle (2.5pt);
\fill (r2) circle (2.5pt);
\fill (r3) circle (2.5pt);
\fill (r4) circle (2.5pt);

\draw[dashed] (c) to (d);
\draw[dashed] (d) to (e);
\draw[dashed] (e) to (f);
\draw[dashed] (f) to (g);

\draw[dashed] (a) to (b) to (c);
\draw[dashed] (g) to (h) to (a);

\draw[color=darksagegreen][color=red] (mb1) to node [midway,fill=white]{4} node[near end,sloped, rotate=90]{$\bowtie$} (i) to node [midway,fill=white]{3} node[near start,sloped, rotate=90]{$\bowtie$} (mt1);
\draw[color=darksagegreen][color=red] (mb2) to node [midway,fill=white] {5} node[near end,sloped, rotate=90]{$\bowtie$} (i) to node [midway,fill=white]{5} node[near start,sloped, rotate=90]{$\bowtie$} (mt2);
\draw[color=darksagegreen][color=red] (mb3) to node [midway,fill=white]{3} node[near end,sloped, rotate=90]{$\bowtie$} (i) to node [midway,fill=white]{4} node[near start,sloped, rotate=90]{$\bowtie$} (mt3);

\coordinate (mt4) at (.9,.23) ;
\coordinate (mb4) at (.9,-.23);
\draw[color=darksagegreen][bend right,color=red] (mt4) to node [left]{5} (mb4);

\draw[color=darksagegreen] (i) to[bend right] node [midway,fill=white]{14} node[near start,sloped, rotate=90]{$\bowtie$} (e);
\draw[color=darksagegreen] (g) to node [midway,fill=white]{22} node[near end,sloped, rotate=90]{$\bowtie$} (i);
\draw[color=darksagegreen] (i) to node [midway,fill=white]{9} node[near start,sloped, rotate=90]{$\bowtie$} (c);
\draw[color=darksagegreen] (i) to[out=135, in = 225, min distance=25mm] node [above left,fill=white]{12} node[near start,sloped, rotate=90]{$\bowtie$} (c);
\draw[color=darksagegreen][color=red] (i) to node[near start,sloped, rotate=90]{$\bowtie$} node[near end,sloped, rotate=90]{$\bowtie$} (q5);
\draw[color=darksagegreen] (q5) to node [midway,fill=white]{10} node[near start,sloped, rotate=90]{$\bowtie$} (c);
%%\draw[color=darksagegreen] (e) to node [midway,fill=white]{13} (c);
\draw[color=darksagegreen] (f) to[out=125,in=180] node [midway,fill=white]{28} node[near end,sloped, rotate=90]{$\bowtie$} (q2);
\draw[color=darksagegreen] (f) to[out=45, in = 0] node [midway,fill=white]{30} node[near end,sloped, rotate=90]{$\bowtie$} (q2);
\draw[color=darksagegreen] (i) to node [midway,fill=white]{6} node[near start,sloped, rotate=90]{$\bowtie$} (h);
\draw[color=darksagegreen] (i) to node [midway,fill=white]{7} node[near start,sloped, rotate=90]{$\bowtie$} (b);
\draw[color=darksagegreen] (q2) to node [midway,fill=white]{15} node[near start,sloped, rotate=90]{$\bowtie$} (e);
\draw[color=darksagegreen][color=red] (q2) to node[near start,sloped, rotate=90]{$\bowtie$} node[near end,sloped, rotate=90]{$\bowtie$} (i);
\draw[color=darksagegreen] [color=red](q3) to node[near start,sloped, rotate=90]{$\bowtie$} node[near end,sloped, rotate=90]{$\bowtie$} (i);
\draw[color=darksagegreen][color=red] (q3) to node[near start,sloped, rotate=90]{$\bowtie$} node[near end,sloped, rotate=90]{$\bowtie$} (q4);
\draw[color=darksagegreen][color=red] (q3) to[out=155,in=0] node[pos=.1,sloped, rotate=90]{$\bowtie$} node[near end,sloped, rotate=90]{$\bowtie$} (q2);
\draw[color=darksagegreen][color=red] (q4) to[out=135,in=0] node[pos=.1,sloped, rotate=90]{$\bowtie$}  node[near end,sloped, rotate=90]{$\bowtie$} (q2);
\draw[color=darksagegreen] (q4) to[out=235,in=45] node [midway,fill=white]{19} node[near start,sloped, rotate=90]{$\bowtie$} (f);
\draw[color=darksagegreen][color=red] (q4) to node[near start,sloped, rotate=90]{$\bowtie$} node[near end,sloped, rotate=90]{$\bowtie$} (i);
\draw[color=darksagegreen](i) to[out=210,in=30] node [midway,fill=white]{8} node[near start,sloped, rotate=90]{$\bowtie$} (f);
\draw[color=darksagegreen] (g) to[out=170,in=200] node [midway,fill=white]{23} (i);

\draw[color=darksagegreen][color=red] (q2) to node[near start,sloped, rotate=90]{$\bowtie$} node[near end,sloped, rotate=90]{$\bowtie$} (r1);
\draw[color=red] (q2) to node[near start,sloped, rotate=90]{$\bowtie$} node[near end,sloped, rotate=90]{$\bowtie$} (r2);
\draw[color=red] (q2) to[out=-15,in = 0,min distance=12mm] node[near start,sloped, rotate=90]{$\bowtie$} node[near end,sloped, rotate=90]{$\bowtie$} (r3);
\draw[color=darksagegreen][color=red] (q2) to[out = 195, in = 180,min distance=12mm] node[near start,sloped, rotate=90]{$\bowtie$} node[near end,sloped, rotate=90]{$\bowtie$}(r3);
\draw[color=darksagegreen][color=red] (r2) to node[near start,sloped, rotate=90]{$\bowtie$} node[near end,sloped, rotate=90]{$\bowtie$}(r1);
\draw[color=darksagegreen][color=red] (r1) to node[near start,sloped, rotate=90]{$\bowtie$} node[near end,sloped, rotate=90]{$\bowtie$}(r3);
\draw[color=darksagegreen][color=red] (r2) to node[near start,sloped, rotate=90]{$\bowtie$} node[near end,sloped, rotate=90]{$\bowtie$} (r3);

\draw[color=darksagegreen][color=red] (q2) to[out=0,in = 0,min distance=14mm] node[near start,sloped, rotate=90]{$\bowtie$} (-1,-.5);
\draw[color=darksagegreen][color=red] (q2) to[out = 180, in = 180,min distance=14mm] node[near start,sloped, rotate=90]{$\bowtie$} (-1,-.5);
\draw[color=darksagegreen][color=red] (-1,-.5) node [fill=white]{29};
\draw[color=darksagegreen] (g) to node [midway,fill=white]{25} (r4);
\draw[color=darksagegreen][color=red] (i) to node [near start,sloped,rotate=90]{$\bowtie$} node [midway,fill=white]{24} (r4);

\draw (a) node[right]{{\small $X$}};
\draw (d) node[left]{{\small $X$}};
\draw (b) node[right]{{\small $X$}};
\draw (c) node[right]{{\small $X$}};
\draw (e) node[left]{{\small $X$}};
\draw (f) node[left]{{\small $X$}};
\draw (g) node[right]{{\small $X$}};
\draw (h) node[right]{{\small $X$}};

\draw (.2,0) node {$P_1$};
\draw (q2) node[above] {$P_2$};
\draw (q3) node[above] {$P_3$};
\draw (q4) node[below right] {$P_4$};
\draw (q5) node[above left] {$P_5$};
\draw (r4) node[left]{$S_1$};

\coordinate (t2) at (-.5,1);
\coordinate (b2) at (-1.3,-.5);
%\fill (t2) circle (1.5pt);
%\fill (b2) circle (1.5pt);

\draw[color=darksagegreen] (e) to node[fill=white]{2} (t2);
\draw[color=darksagegreen] (b2) to[out=90,in=180] node[near start, fill=white]{2} node[near end,sloped, rotate=90]{$\bowtie$} (q2);

\draw[color=darksagegreen][color=red] (i) to[out=150,in=-90] node[pos=.85, fill=white]{13} node[near start,sloped, rotate=90]{$\bowtie$} (-.35,1);
\draw[color=darksagegreen][color=red] (-1.2,-.6666667) to[out=95, in =180,min distance = 16 mm] (q2);

\draw[color=darksagegreen][color=red] (i) to[out=210,in=80] node [near end,fill=white]{1} node[near start,sloped, rotate=90]{$\bowtie$}(-0.5,-1);
\draw[color=darksagegreen][color=red] (-1.3,.5) to node [fill=white]{1} (-.75,1);
\draw[color=darksagegreen][color=red] (-1.4,-.333) to[out=90,in=180] node [near start, fill=white]{1} node[near end,sloped, rotate=90]{$\bowtie$} (q2); 

\end{tikzpicture}

\caption{The top triangulation is $\musep^{\mathcal{M}_1,*}\mucycle^{\mathcal{M}_1}e\musep^{\mathcal{M}_1}\musep^{\mathcal{M}_0,*}\mucycle^{\mathcal{M}_0}\musep^{\mathcal{M}_0}(T^*)$. The bottom trianglation is $\musep^X\musep^{\mathcal{M}_1,*}\mucycle^{\mathcal{M}_1}e\musep^{\mathcal{M}_1}\musep^{\mathcal{M}_0,*}\mucycle^{\mathcal{M}_0}\musep^{\mathcal{M}_0}(T^*)$, where $\musep^X$ is defined in Example \ref{ex:main_6}.}\label{fig:main_example_triangulation_5}
\end{figure}
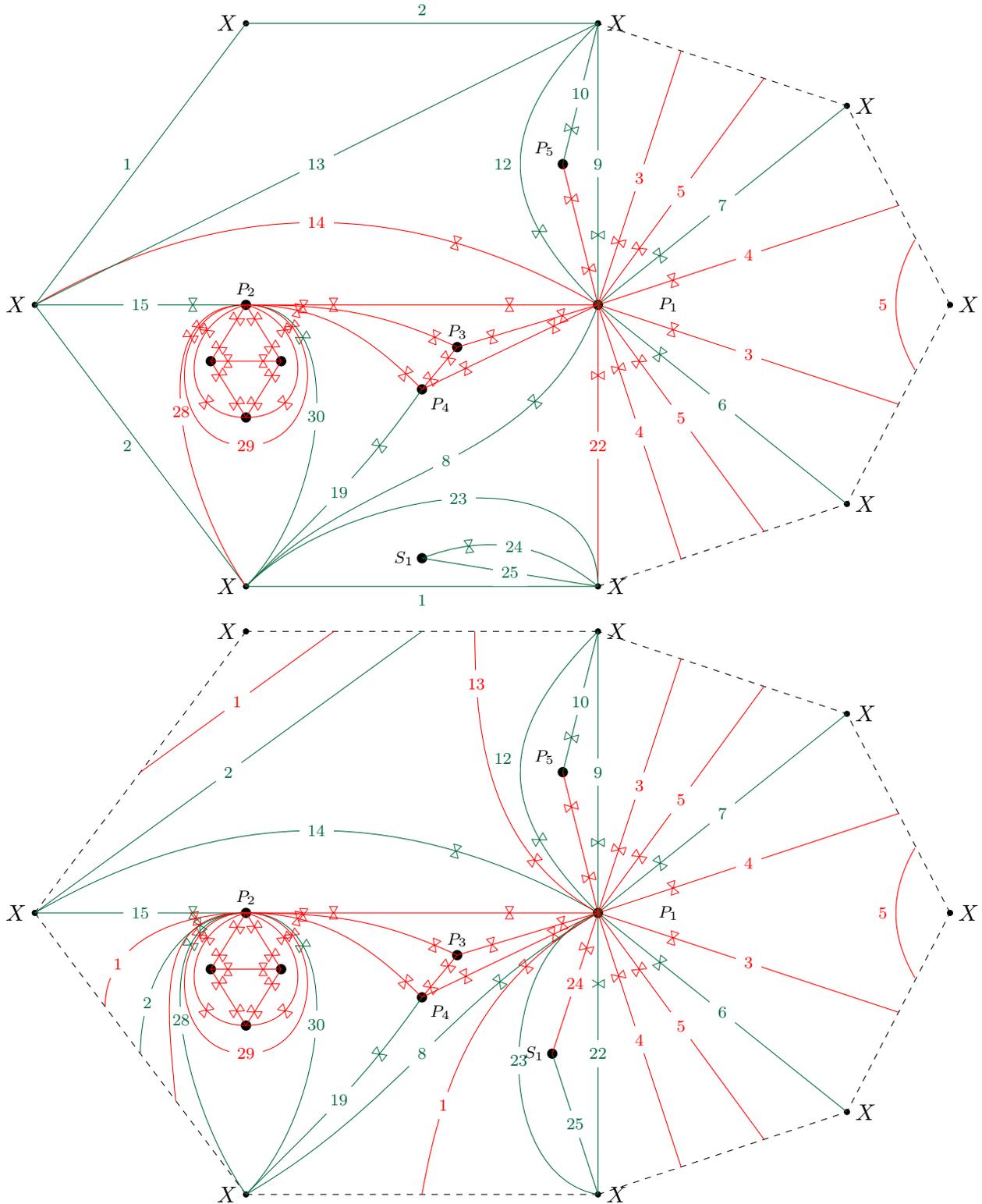

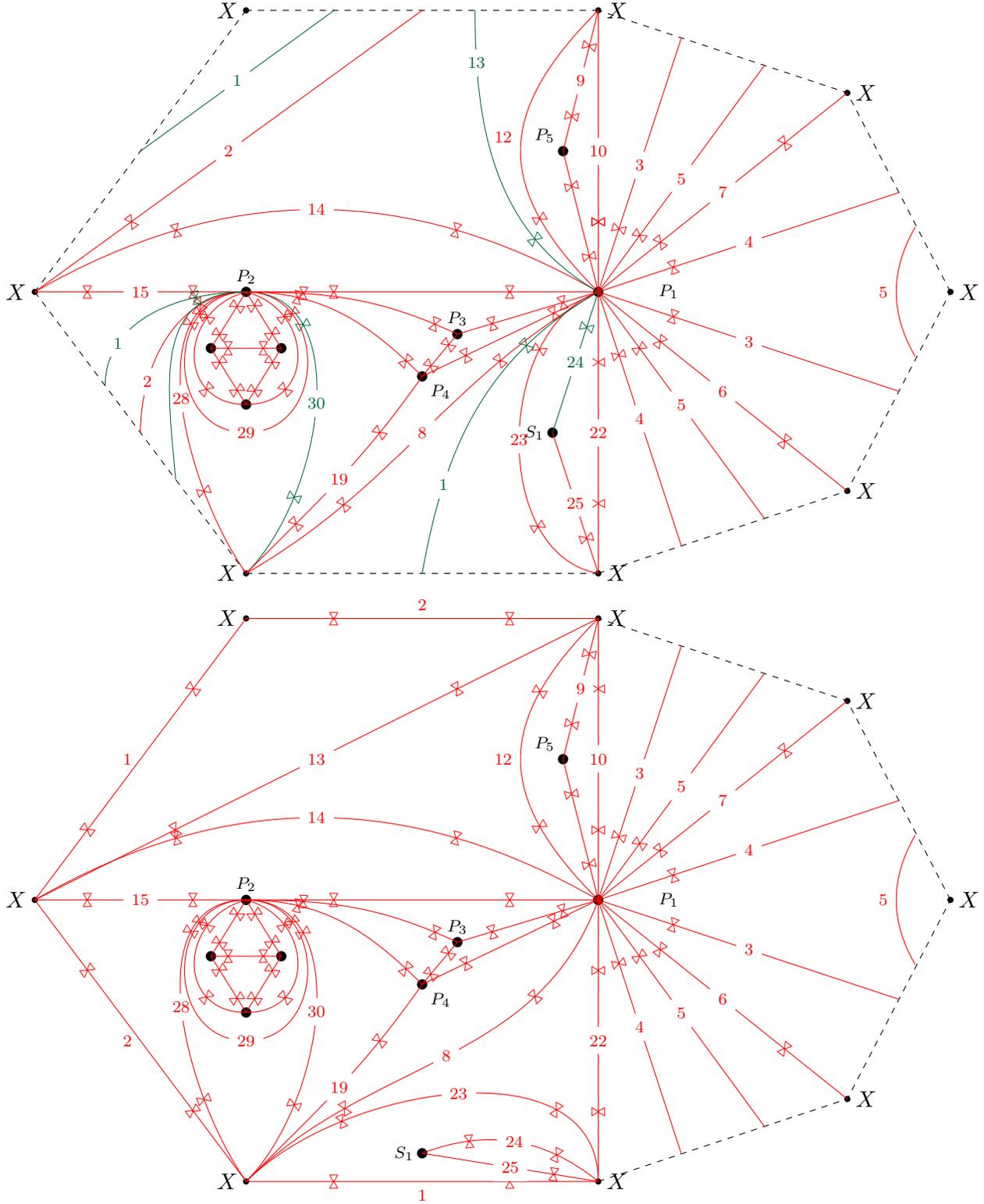
\begin{figure}[t]
\centering
\begin{tikzpicture}[x=6cm,y=4.8cm,font=\tiny]
\coordinate (a) at (1.00000000000000,0.000000000000000);
\coordinate (b) at (0.707106781186548,0.707106781186548);
\coordinate (c) at (0.00000000000000,1.00000000000000);
\coordinate (d) at (-1,1);
\coordinate (e) at (-1.6,0);
\coordinate (f) at (-1,-1);
\coordinate (g) at (0,-1.00000000000000);
\coordinate (h) at (0.707106781186548,-0.707106781186548);
\coordinate (i) at (0.0,0);
\coordinate (q5) at (-.1,0.5);
\coordinate (q2) at (-1,0);
\coordinate (q3) at (-.4,-.15);
\coordinate (q4) at (-.5,-.3);
\coordinate (r1) at (-1.1,-.2);
\coordinate (r2) at (-.9,-.2);
\coordinate (r3) at (-1,-.4);
\coordinate (r4) at (-.13,-.5);

\coordinate (mb1) at (.2357,.098-1);
\coordinate (mb2) at (2*.2357,2*.098-1);
\coordinate (mt1) at (.2357,-.098+1);
\coordinate (mt2) at (2*.2357,-2*.098+1);
\coordinate (mt3) at (.855,0.707106781186548/2);
\coordinate (mb3) at (.855,-0.707106781186548/2);

\fill (a) circle (1.5pt);
\fill (b) circle (1.5pt);
\fill (c) circle (1.5pt);
\fill (d) circle (1.5pt);
\fill (e) circle (1.5pt);
\fill (f) circle (1.5pt);
\fill (g) circle (1.5pt);
\fill (h) circle (1.5pt);
\fill (i) circle (2.5pt);
\fill (q5) circle (2.5pt);
\fill (q4) circle (2.5pt);
\fill (q3) circle (2.5pt);
\fill (q2) circle (2.5pt);
\fill (r1) circle (2.5pt);
\fill (r2) circle (2.5pt);
\fill (r3) circle (2.5pt);
\fill (r4) circle (2.5pt);

\draw[dashed] (c) to (d);
\draw[dashed] (d) to (e);
\draw[dashed] (e) to (f);
\draw[dashed] (f) to (g);

\draw[dashed] (a) to (b) to (c);
\draw[dashed] (g) to (h) to (a);

\draw[color=darksagegreen][color=red] (mb1) to node [midway,fill=white]{4} node[near end,sloped, rotate=90]{$\bowtie$} (i) to node [midway,fill=white]{3} node[near start,sloped, rotate=90]{$\bowtie$} (mt1);
\draw[color=darksagegreen][color=red] (mb2) to node [midway,fill=white] {5} node[near end,sloped, rotate=90]{$\bowtie$} (i) to node [midway,fill=white]{5} node[near start,sloped, rotate=90]{$\bowtie$} (mt2);
\draw[color=darksagegreen][color=red] (mb3) to node [midway,fill=white]{3} node[near end,sloped, rotate=90]{$\bowtie$} (i) to node [midway,fill=white]{4} node[near start,sloped, rotate=90]{$\bowtie$} (mt3);

\coordinate (mt4) at (.9,.23) ;
\coordinate (mb4) at (.9,-.23);
\draw[color=darksagegreen][bend right,color=red] (mt4) to node [left]{5} (mb4);

\draw[color=darksagegreen][color=red] (i) to[bend right] node [midway,fill=white]{14} node[near start,sloped, rotate=90]{$\bowtie$}  node[near end,sloped, rotate=90]{$\bowtie$}(e);
\draw[color=darksagegreen][color=red] (g) to node [midway,fill=white]{22} node[near end,sloped, rotate=90]{$\bowtie$} node[near start,sloped, rotate=90]{$\bowtie$} (i);
\draw[color=darksagegreen][color=red] (i) to node [midway,fill=white]{10} node[near start,sloped, rotate=90]{$\bowtie$} node[near start,sloped, rotate=90]{$\bowtie$} (c);
\draw[color=darksagegreen][color=red] (i) to[out=135, in = 225, min distance=25mm] node [above left,fill=white]{12} node[near start,sloped, rotate=90]{$\bowtie$} (c);
\draw[color=darksagegreen][color=red] (i) to node[near start,sloped, rotate=90]{$\bowtie$} node[near end,sloped, rotate=90]{$\bowtie$} (q5);
\draw[color=darksagegreen][color=red] (q5) to node [midway,fill=white]{9} node[near start,sloped, rotate=90]{$\bowtie$} node[near end,sloped, rotate=90]{$\bowtie$}(c);
%%\draw[color=darksagegreen] (e) to node [midway,fill=white]{13} (c);
\draw[color=darksagegreen][color=red] (f) to[out=125,in=180] node [midway,fill=white]{28} node[near start,sloped, rotate=90]{$\bowtie$} node[near end,sloped, rotate=90]{$\bowtie$} (q2);
\draw[color=darksagegreen] (f) to[out=45, in = 0] node [midway,fill=white]{30} node[near end,sloped, rotate=90]{$\bowtie$} node[near start,sloped, rotate=90]{$\bowtie$} (q2);
\draw[color=darksagegreen][color=red] (i) to node [midway,fill=white]{6} node[near start,sloped, rotate=90]{$\bowtie$} node[near end,sloped, rotate=90]{$\bowtie$}(h);
\draw[color=darksagegreen][color=red] (i) to node [midway,fill=white]{7} node[near start,sloped, rotate=90]{$\bowtie$} node[near end,sloped, rotate=90]{$\bowtie$} (b);
\draw[color=darksagegreen][color=red] (q2) to node [midway,fill=white]{15} node[near start,sloped, rotate=90]{$\bowtie$} node[near end,sloped, rotate=90]{$\bowtie$} (e);
\draw[color=darksagegreen][color=red] (q2) to node[near start,sloped, rotate=90]{$\bowtie$} node[near end,sloped, rotate=90]{$\bowtie$} (i);
\draw[color=darksagegreen][color=red](q3) to node[near start,sloped, rotate=90]{$\bowtie$} node[near end,sloped, rotate=90]{$\bowtie$} (i);
\draw[color=darksagegreen][color=red] (q3) to node[near start,sloped, rotate=90]{$\bowtie$} node[near end,sloped, rotate=90]{$\bowtie$} (q4);
\draw[color=darksagegreen][color=red] (q3) to[out=155,in=0] node[pos=.1,sloped, rotate=90]{$\bowtie$} node[near end,sloped, rotate=90]{$\bowtie$} (q2);
\draw[color=darksagegreen][color=red] (q4) to[out=135,in=0] node[pos=.1,sloped, rotate=90]{$\bowtie$}  node[near end,sloped, rotate=90]{$\bowtie$} (q2);
\draw[color=darksagegreen][color=red] (q4) to[out=235,in=45] node [midway,fill=white]{19}  node[near start,sloped, rotate=90]{$\bowtie$} node[near end,sloped, rotate=90]{$\bowtie$} (f);
\draw[color=darksagegreen][color=red] (q4) to node[near start,sloped, rotate=90]{$\bowtie$} node[near end,sloped, rotate=90]{$\bowtie$} (i);
\draw[color=darksagegreen][color=red] (i) to[out=210,in=30] node [midway,fill=white]{8} node[near start,sloped, rotate=90]{$\bowtie$} node[near end,sloped, rotate=90]{$\bowtie$} (f);
\draw[color=darksagegreen][color=red] (g) to[out=170,in=200] node [midway,fill=white]{23} node[near start,sloped, rotate=90]{$\bowtie$} node[near end,sloped, rotate=90]{$\bowtie$} (i);

\draw[color=darksagegreen][color=red] (q2) to node[near start,sloped, rotate=90]{$\bowtie$} node[near end,sloped, rotate=90]{$\bowtie$} (r1);
\draw[color=red] (q2) to node[near start,sloped, rotate=90]{$\bowtie$} node[near end,sloped, rotate=90]{$\bowtie$} (r2);
\draw[color=red] (q2) to[out=-15,in = 0,min distance=12mm] node[near start,sloped, rotate=90]{$\bowtie$} node[near end,sloped, rotate=90]{$\bowtie$} (r3);
\draw[color=darksagegreen][color=red] (q2) to[out = 195, in = 180,min distance=12mm] node[near start,sloped, rotate=90]{$\bowtie$} node[near end,sloped, rotate=90]{$\bowtie$}(r3);
\draw[color=darksagegreen][color=red] (r2) to node[near start,sloped, rotate=90]{$\bowtie$} node[near end,sloped, rotate=90]{$\bowtie$}(r1);
\draw[color=darksagegreen][color=red] (r1) to node[near start,sloped, rotate=90]{$\bowtie$} node[near end,sloped, rotate=90]{$\bowtie$}(r3);
\draw[color=darksagegreen][color=red] (r2) to node[near start,sloped, rotate=90]{$\bowtie$} node[near end,sloped, rotate=90]{$\bowtie$} (r3);

\draw[color=darksagegreen][color=red] (q2) to[out=0,in = 0,min distance=14mm] node[near start,sloped, rotate=90]{$\bowtie$} (-1,-.5);
\draw[color=darksagegreen][color=red] (q2) to[out = 180, in = 180,min distance=14mm] node[near start,sloped, rotate=90]{$\bowtie$} (-1,-.5);
\draw[color=darksagegreen][color=red] (-1,-.5) node [fill=white]{29};
\draw[color=darksagegreen][color=red] (g) to node [midway,fill=white]{25}  node[near start,sloped, rotate=90]{$\bowtie$} (r4);
\draw[color=darksagegreen] (i) to node [near start,sloped,rotate=90]{$\bowtie$} node [midway,fill=white]{24} (r4);

\draw (a) node[right]{{\small $X$}};
\draw (d) node[left]{{\small $X$}};
\draw (b) node[right]{{\small $X$}};
\draw (c) node[right]{{\small $X$}};
\draw (e) node[left]{{\small $X$}};
\draw (f) node[left]{{\small $X$}};
\draw (g) node[right]{{\small $X$}};
\draw (h) node[right]{{\small $X$}};

\draw (.2,0) node {$P_1$};
\draw (q2) node[above] {$P_2$};
\draw (q3) node[above] {$P_3$};
\draw (q4) node[below right] {$P_4$};
\draw (q5) node[above left] {$P_5$};
\draw (r4) node[left]{$S_1$};

\coordinate (t2) at (-.5,1);
\coordinate (b2) at (-1.3,-.5);
%\fill (t2) circle (1.5pt);
%\fill (b2) circle (1.5pt);

\draw[color=darksagegreen][color=red] (e) to node[fill=white]{2} node[near start,sloped, rotate=90]{$\bowtie$} (t2);
\draw[color=darksagegreen][color=red] (b2) to[out=90,in=180] node[near start, fill=white]{2} node[near end,sloped, rotate=90]{$\bowtie$} (q2);

\draw[color=darksagegreen] (i) to[out=150,in=-90] node[pos=.85, fill=white]{13} node[near start,sloped, rotate=90]{$\bowtie$} (-.35,1);
\draw[color=darksagegreen] (-1.2,-.6666667) to[out=95, in =180,min distance = 16 mm] (q2);

\draw[color=darksagegreen] (i) to[out=210,in=80] node [near end,fill=white]{1} node[near start,sloped, rotate=90]{$\bowtie$}(-0.5,-1);
\draw[color=darksagegreen] (-1.3,.5) to node [fill=white]{1} (-.75,1);
\draw[color=darksagegreen] (-1.4,-.333) to[out=90,in=180] node [near start, fill=white]{1} node[near end,sloped, rotate=90]{$\bowtie$} (q2); 

\end{tikzpicture}

\begin{tikzpicture}[x=6cm,y=4.8cm,font=\tiny]
\coordinate (a) at (1.00000000000000,0.000000000000000);
\coordinate (b) at (0.707106781186548,0.707106781186548);
\coordinate (c) at (0.00000000000000,1.00000000000000);
\coordinate (d) at (-1,1);
\coordinate (e) at (-1.6,0);
\coordinate (f) at (-1,-1);
\coordinate (g) at (0,-1.00000000000000);
\coordinate (h) at (0.707106781186548,-0.707106781186548);
\coordinate (i) at (0.0,0);
\coordinate (q5) at (-.1,0.5);
\coordinate (q2) at (-1,0);
\coordinate (q3) at (-.4,-.15);
\coordinate (q4) at (-.5,-.3);
\coordinate (r1) at (-1.1,-.2);
\coordinate (r2) at (-.9,-.2);
\coordinate (r3) at (-1,-.4);
\coordinate (r4) at (-.5,-.9);

\coordinate (mb1) at (.2357,.098-1);
\coordinate (mb2) at (2*.2357,2*.098-1);
\coordinate (mt1) at (.2357,-.098+1);
\coordinate (mt2) at (2*.2357,-2*.098+1);
\coordinate (mt3) at (.855,0.707106781186548/2);
\coordinate (mb3) at (.855,-0.707106781186548/2);

\fill (a) circle (1.5pt);
\fill (b) circle (1.5pt);
\fill (c) circle (1.5pt);
\fill (d) circle (1.5pt);
\fill (e) circle (1.5pt);
\fill (f) circle (1.5pt);
\fill (g) circle (1.5pt);
\fill (h) circle (1.5pt);
\fill (i) circle (2.5pt);
\fill (q5) circle (2.5pt);
\fill (q4) circle (2.5pt);
\fill (q3) circle (2.5pt);
\fill (q2) circle (2.5pt);
\fill (r1) circle (2.5pt);
\fill (r2) circle (2.5pt);
\fill (r3) circle (2.5pt);
\fill (r4) circle (2.5pt);

\draw[color=darksagegreen][color=red] (c) to node [above] {2} node[near start,sloped, rotate=90]{$\bowtie$} node[near end,sloped, rotate=90]{$\bowtie$} (d) to node [left] {1} node[near start,sloped, rotate=90]{$\bowtie$} node[near end,sloped, rotate=90]{$\bowtie$}(e) to node [left] {2}node[near end,sloped, rotate=90]{$\bowtie$} node[near start,sloped, rotate=90]{$\bowtie$} (f) to node [below] {1} node[near start,sloped, rotate=90]{$\bowtie$} node[near end,sloped, rotate=90]{$\bowtie$}(g);

\draw[dashed] (a) to (b) to (c);
\draw[dashed] (g) to (h) to (a);

\draw[color=darksagegreen][color=red] (mb1) to node [midway,fill=white]{4} node[near end,sloped, rotate=90]{$\bowtie$} (i) to node [midway,fill=white]{3} node[near start,sloped, rotate=90]{$\bowtie$} (mt1);
\draw[color=darksagegreen][color=red] (mb2) to node [midway,fill=white] {5} node[near end,sloped, rotate=90]{$\bowtie$} (i) to node [midway,fill=white]{5} node[near start,sloped, rotate=90]{$\bowtie$} (mt2);
\draw[color=darksagegreen][color=red] (mb3) to node [midway,fill=white]{3} node[near end,sloped, rotate=90]{$\bowtie$} (i) to node [midway,fill=white]{4} node[near start,sloped, rotate=90]{$\bowtie$} (mt3);

\coordinate (mt4) at (.9,.23) ;
\coordinate (mb4) at (.9,-.23);
\draw[color=darksagegreen][bend right,color=red] (mt4) to node [left]{5} (mb4);

\draw[color=darksagegreen][color=red] (i) to[bend right] node [midway,fill=white]{14} node[near start,sloped, rotate=90]{$\bowtie$} node[near end,sloped, rotate=90]{$\bowtie$} (e);
\draw[color=darksagegreen][color=red] (g) to node [midway,fill=white]{22} node[near start,sloped, rotate=90]{$\bowtie$} node[near end,sloped, rotate=90]{$\bowtie$} (i);
\draw[color=darksagegreen][color=red] (i) to node [midway,fill=white]{10} node[near start,sloped, rotate=90]{$\bowtie$} node[near end,sloped, rotate=90]{$\bowtie$}(c);
\draw[color=darksagegreen][color=red] (i) to[out=135, in = 225, min distance=25mm] node [left,fill=white]{12} node[near start,sloped, rotate=90]{$\bowtie$}node[near end,sloped, rotate=90]{$\bowtie$} (c);
\draw[color=darksagegreen][color=red] (i) to node[near start,sloped, rotate=90]{$\bowtie$} node[near end,sloped, rotate=90]{$\bowtie$} (q5);
\draw[color=darksagegreen][color=red] (q5) to node [midway,fill=white]{9} node[near start,sloped, rotate=90]{$\bowtie$} node[near end,sloped, rotate=90]{$\bowtie$} (c);
\draw[color=darksagegreen][color=red] (e) to node [midway,fill=white]{13} node[near start,sloped, rotate=90]{$\bowtie$} node[near end,sloped, rotate=90]{$\bowtie$} (c);
\draw[color=darksagegreen][color=red] (f) to[out=125,in=180] node [midway,fill=white]{28} node[near end,sloped, rotate=90]{$\bowtie$} node[near start,sloped, rotate=90]{$\bowtie$}(q2);
\draw[color=darksagegreen][color=red] (f) to[out=45, in = 0] node [midway,fill=white]{30} node[near end,sloped, rotate=90]{$\bowtie$} node[near start,sloped, rotate=90]{$\bowtie$}(q2);
\draw[color=darksagegreen][color=red] (i) to node [midway,fill=white]{6} node[near start,sloped, rotate=90]{$\bowtie$} node[near end,sloped, rotate=90]{$\bowtie$}(h);
\draw[color=darksagegreen][color=red] (i) to node [midway,fill=white]{7} node[near start,sloped, rotate=90]{$\bowtie$} node[near end,sloped, rotate=90]{$\bowtie$} (b);
\draw[color=darksagegreen][color=red] (q2) to node [midway,fill=white]{15} node[near end,sloped, rotate=90]{$\bowtie$} node[near start,sloped, rotate=90]{$\bowtie$} (e);
\draw[color=darksagegreen][color=red] (q2) to node[near start,sloped, rotate=90]{$\bowtie$} node[near end,sloped, rotate=90]{$\bowtie$} (i);
\draw[color=darksagegreen] [color=red](q3) to node[near start,sloped, rotate=90]{$\bowtie$} node[near end,sloped, rotate=90]{$\bowtie$} (i);
\draw[color=darksagegreen][color=red] (q3) to node[near start,sloped, rotate=90]{$\bowtie$} node[near end,sloped, rotate=90]{$\bowtie$} (q4);
\draw[color=darksagegreen][color=red] (q3) to[out=155,in=0] node[pos=.1,sloped, rotate=90]{$\bowtie$} node[near end,sloped, rotate=90]{$\bowtie$} (q2);
\draw[color=darksagegreen][color=red] (q4) to[out=135,in=0] node[pos=.1,sloped, rotate=90]{$\bowtie$}  node[near end,sloped, rotate=90]{$\bowtie$} (q2);
\draw[color=darksagegreen][color=red] (q4) to[out=235,in=45] node [midway,fill=white]{19} node[near start,sloped, rotate=90]{$\bowtie$} (f);
\draw[color=darksagegreen][color=red] (q4) to node[near start,sloped, rotate=90]{$\bowtie$} node[near end,sloped, rotate=90]{$\bowtie$} (i);
\draw[color=darksagegreen][color=red] (i) to[out=250,in=45] node [midway,fill=white]{8} node[near start,sloped, rotate=90]{$\bowtie$} node[near end,sloped, rotate=90]{$\bowtie$} (f);
\draw[color=darksagegreen][color=red] (g) to[out=90,in=45] node [midway,fill=white]{23} node[near start,sloped, rotate=90]{$\bowtie$} node[near end,sloped, rotate=90]{$\bowtie$} (f);

\draw[color=darksagegreen][color=red] (q2) to node[near start,sloped, rotate=90]{$\bowtie$} node[near end,sloped, rotate=90]{$\bowtie$} (r1);
\draw[color=red] (q2) to node[near start,sloped, rotate=90]{$\bowtie$} node[near end,sloped, rotate=90]{$\bowtie$} (r2);
\draw[color=red] (q2) to[out=-15,in = 0,min distance=12mm] node[near start,sloped, rotate=90]{$\bowtie$} node[near end,sloped, rotate=90]{$\bowtie$} (r3);
\draw[color=darksagegreen][color=red] (q2) to[out = 195, in = 180,min distance=12mm] node[near start,sloped, rotate=90]{$\bowtie$} node[near end,sloped, rotate=90]{$\bowtie$}(r3);
\draw[color=darksagegreen][color=red] (r2) to node[near start,sloped, rotate=90]{$\bowtie$} node[near end,sloped, rotate=90]{$\bowtie$}(r1);
\draw[color=darksagegreen][color=red] (r1) to node[near start,sloped, rotate=90]{$\bowtie$} node[near end,sloped, rotate=90]{$\bowtie$}(r3);
\draw[color=darksagegreen][color=red] (r2) to node[near start,sloped, rotate=90]{$\bowtie$} node[near end,sloped, rotate=90]{$\bowtie$} (r3);

\draw[color=darksagegreen][color=red] (q2) to[out=0,in = 0,min distance=14mm] node[near start,sloped, rotate=90]{$\bowtie$} (-1,-.5);
\draw[color=darksagegreen][color=red] (q2) to[out = 180, in = 180,min distance=14mm] node[near start,sloped, rotate=90]{$\bowtie$} (-1,-.5);
\draw[color=darksagegreen][color=red] (-1,-.5) node [fill=white]{29};
\draw[color=darksagegreen][color=red] (g) to node [midway,fill=white]{25} node[near start,sloped, rotate=90]{$\bowtie$}(r4);
\draw[color=darksagegreen][color=red] (g) to[bend right] node [near end,sloped,rotate=90]{$\bowtie$} node [midway,fill=white]{24} node[near start,sloped, rotate=90]{$\bowtie$}(r4);

\draw (a) node[right]{{\small $X$}};
\draw (d) node[left]{{\small $X$}};
\draw (b) node[right]{{\small $X$}};
\draw (c) node[right]{{\small $X$}};
\draw (e) node[left]{{\small $X$}};
\draw (f) node[left]{{\small $X$}};
\draw (g) node[right]{{\small $X$}};
\draw (h) node[right]{{\small $X$}};

\draw (.2,0) node {$P_1$};
\draw (q2) node[above] {$P_2$};
\draw (q3) node[above] {$P_3$};
\draw (q4) node[below right] {$P_4$};
\draw (q5) node[above left] {$P_5$};
\draw (r4) node[left]{$S_1$};
\end{tikzpicture}
\caption{The top triangulation is $\mucycle^X\musep^X\musep^{\mathcal{M}_1,*}\mucycle^{\mathcal{M}_1}e\musep^{\mathcal{M}_1}\musep^{\mathcal{M}_0,*}\mucycle^{\mathcal{M}_0}\musep^{\mathcal{M}_0}(T^*)$, and the bottom triangulation is $\musep^{X,*}\mucycle^X\musep^X\musep^{\mathcal{M}_1,*}\mucycle^{\mathcal{M}_1}e\musep^{\mathcal{M}_1}\musep^{\mathcal{M}_0,*}\mucycle^{\mathcal{M}_0}\musep^{\mathcal{M}_0}(T^*)$. Both $\musep^X$ and $\mucycle^X$ are defined in Example \ref{ex:main_6}. It is clear from the picture that this is a maximal green sequence for $T^*.$}\label{fig:main_example_triangulation_6}
\end{figure}

\end{document}